\documentclass[a4paper, 11pt, reqno]{amsart}

\usepackage{mleftright}
\usepackage[utf8]{inputenc}
\usepackage[english]{babel}
\usepackage[T1]{fontenc}
\usepackage[left=3cm,right=3cm,top=3cm,bottom=3cm]{geometry}
\usepackage{amsmath}
\usepackage{float}
\usepackage{booktabs}
\usepackage{amssymb}
\usepackage{amsfonts}
\usepackage{mathtools}
\usepackage{latexsym}
\usepackage{graphicx}
\usepackage{pdflscape}
\usepackage{gensymb}
\usepackage{trsym}
\usepackage{url}
\usepackage{amsthm}
\usepackage{tikz-cd}
\usepackage{hyperref}
\usepackage{environ}
\usepackage[most]{tcolorbox}
\tcbuselibrary{theorems}

\usepackage{cleveref}

\newtheorem{prop}{Proposition}[section]

\theoremstyle{remark}
\newtheorem{remark}[prop]{Remark}

\newtcbtheorem[use counter=prop, number within=section, crefname={definition}{definition}, Crefname={Definition}{definition}]
{definition}{Definition}{
	breakable,
	fonttitle = \bfseries,
	colframe = gray!75!black,
	colback = gray!10
}
{def}

\newtcbtheorem[use counter=prop, number within=section, crefname={theorem}{theorem}, Crefname={Theorem}{theorem}]
{theorem}{Theorem}{
	breakable,
	fonttitle = \bfseries,
	colframe = gray!75!black,
	colback = gray!10
}
{thm}

\newtcbtheorem[use counter=prop, number within=section, crefname={lemma}{lemma}, Crefname={Lemma}{lemma}]
{lemma}{Lemma}{
    breakable,
    enhanced,
    fonttitle = \bfseries,
    colframe = gray!5!black,
    colback = gray!15
}
{lem}

\tcbset{boxrule=1.4pt,boxsep=1pt,top=2pt,bottom=1pt}

\newcommand{\set}[2]{\mleft\{ #1 \,\middle|\, #2 \mright\}}
\newcommand{\restr}[2][{}]{\,{\big|}^{#1}_{#2}}
\newcommand{\defgl}{\mathrel{=\!\!\mathop:}}
\newcommand{\defgr}{\mathrel{\mathop:\!\!=}}
\newcommand{\C}{\mathbb{C}}
\newcommand{\R}{\mathbb{R}}

\begin{document}

\title[Meromorphic Continuation of Weighted Zetas]{Meromorphic Continuation of Weighted Zeta Functions on Open Hyperbolic Systems}
\author[P. Schütte and T. Weich]%
{Philipp Schütte and Tobias Weich\vskip.025in
with an appendix by Sonja Barkhofen, Philipp Schütte and Tobias Weich}

\email{pschuet2@mail.uni-paderborn.de}
\email{weich@math.uni-paderborn.de}
\address{Institute of Mathematics, Paderborn University, Paderborn, Germany}
\email{sonja.barkhofen@uni-paderborn.de}
\address{Department of Physics, Paderborn University, Paderborn, Germany}
\begin{abstract}
	In this article we prove meromorphic continuation of weighted zeta functions $Z_f$ in the framework of open hyperbolic systems by using the meromorphically continued restricted resolvent of Dyatlov and Guillarmou \cite{Dyatlov.2016a}. We obtain a residue formula proving equality between residues of $Z_f$ and invariant Ruelle distributions. We combine this equality with results of Guillarmou, Hilgert and Weich \cite{Weich.2021} in order to relate the residues to Patterson-Sullivan distributions. Finally we provide proof-of-principle results concerning the numerical calculation of invariant Ruelle distributions for $3$-disc scattering systems.
\end{abstract}

\maketitle	
	
\tableofcontents
	
	
\section{Introduction} \label{intro}

In many applications, dynamical zeta functions provide a convenient way of describing invariants of dynamical systems. In this paper we present an analytical study of new \emph{weighted zeta functions}, which not only describe the Pollicott-Ruelle resonances of a system via their poles but also the invariant Ruelle distributions via their residues. We achieve meromorphic continuation of our weighted zeta functions by using the results of Dyatlov and Guillarmou \cite{Dyatlov.2016a} in the setting of open hyperbolic systems. In an appendix we complement this analytical side with first numerical results for $3$-disc systems which reveal interesting possible applications in the study of phenomena such as recurrence to a neighborhood of the critical line.

\subsection{Statement of results}

To formulate our main results let $\mathcal{M}$ be a smooth manifold with smooth (possibly empty) boundary, $X$ a smooth, nowhere vanishing vector field on $\mathcal{M}$, and $\varphi_t$ its flow. In addition, denote by $\mathcal{E}$ a smooth, complex vector bundle over $\mathcal{M}$ and by $\mathbf{X}$ a first-order differential operator acting on sections of $\mathcal{E}$ which is related to $X$ via the Leibniz rule
\begin{equation*}
\mathbf{X}(f\mathbf{u}) = (Xf) \mathbf{u} + f (\mathbf{X}\mathbf{u}),\qquad f\in\mathrm{C}^\infty(\mathcal{M}),~ \mathbf{u}\in \mathrm{C}^\infty(\mathcal{M}, \mathcal{E}) ~.
\end{equation*}
The dynamics of interest from the point of view of our application happen on the \emph{trapped set $K$} of $\varphi_t$, i.e. on the set
\begin{equation} \label{trapped_set_def}
K \defgr \left\{ x\in\mathcal{M} \,|\, \varphi_t(x) ~\text{defined}~ \forall t\in\mathbb{R} ~\text{and}~ \exists~\text{cpt.}~ A\subseteq\mathcal{M}~\text{with}~ \varphi_t(x)\in A ~\forall t\in\mathbb{R} \right\} .
\end{equation}
Finally, we make the following \emph{dynamical assumptions}:
\begin{enumerate}
	\item $K$ is compact,
	\item $K\subseteq \mathring{\mathcal{M}}$, with $\mathring{\mathcal{M}}$ the manifold interior of $\mathcal{M}$,
	\item $\varphi_t$ is hyperbolic on $K$.
\end{enumerate}
For the formal definition of hyperbolicity see Section \ref{setup}. In the setting just described one can define a discrete subset of the complex plane $\mathbb{C}$ called \textit{Pollicott-Ruelle resonances} of $\mathbf{X}$, directly based on the work of Dyatlov and Guillarmou \cite{Dyatlov.2016a}. They arise as the poles of the meromorphic continuation of the resolvent $(\mathbf{X} + \lambda)^{-1}$ and any such resonance $\lambda_0$ is associated with a finite rank residue operator $\Pi_{\lambda_0}$. It is possible to obtain a precise wavefront set estimate for $\Pi_{\lambda_0}$ which in particular guarantees the existence of the flat trace $\mathrm{tr}^\flat$ and therefore the well-definedness of the \emph{invariant Ruelle distributions} $\mathrm{tr}^\flat\left( \Pi_{\lambda_0} f \right)$. With these ingredients we can formally define our object of main interest, namely the \textit{weighted zeta function} with weight $f\in\mathrm{C}^\infty(\mathcal{M})$ at $\lambda\in \mathbb{C}$:
\begin{equation} \label{eq_zeta_def}
Z_f^\mathbf{X}(\lambda) \defgr \sum_{\gamma} \left( \frac{\exp\left(-\lambda T_\gamma\right) \mathrm{tr}(\alpha_\gamma)}{\vert \det(\mathrm{id} - \mathcal{P}_\gamma) \vert} \int_{\gamma^\#} f \right) ~,
\end{equation}
where the sum is over all closed trajectories $\gamma$ of $\varphi_t$, $T_\gamma$ is its period, $\gamma^\#$ denotes the corresponding primitive closed trajectory, $\mathcal{P}_\gamma$ its linearized Poincar\'{e} map and $\alpha_\gamma$ the parallel transport map in $\mathcal{E}$ associated with $\mathbf{X}$. For formal definitions of these objects we again refer the reader to Section \ref{setup}. Our main result now reads as follows:
	
\begin{theorem}{Meromorphic Continuation of Weighted Zetas I}{analyt_zeta}		
	$Z_f^\mathbf{X}$ converges absolutely in $\{\mathrm{Re}(\lambda) \gg 1\}$ and continues meromorphically to $\mathbb{C}$. Any pole $\lambda_0$ of $Z_f$ is a Pollicott-Ruelle resonance of $\mathbf{X}$ and if the resolvent has a pole of order $J(\lambda_0)$ at $\lambda_0$ then for $k\leq J(\lambda_0)$ we have
	\begin{equation} \label{eq:laurent}
	\mathrm{Res}_{\lambda = \lambda_0} \left[ Z_f^\mathbf{X}(\lambda)(\lambda - \lambda_0)^k \right] = \mathrm{tr}^\flat \left( (\mathbf{X} - \lambda_0)^k \Pi_{\lambda_0} f \right) .
	\end{equation}
\end{theorem}
	
\begin{remark}
	Note that in the particularly simple situation of a trivial, one-dimensional bundle $\mathcal{E} = \mathcal{M} \times \mathbb{C}$ and $\mathbf{X} = X$ our weighted zeta simplifies and is given by the following expression:
	\begin{equation*}
	Z_f(\lambda) \defgr Z_f^\mathbf{X}(\lambda) = \sum_{\gamma} \left( \frac{\exp\left(-\lambda T_\gamma\right)}{\vert \det(\mathrm{id} - \mathcal{P}_\gamma) \vert} \int_{\gamma^\#} f \right) .
	\end{equation*}
\end{remark}
	
Our second main theorem derives from the observation that \Cref{thm:analyt_zeta} equates the residue of $Z_f^\mathbf{X}(\lambda)$, i.e. the case $k = 0$ in \eqref{eq:laurent}, with the so-called \emph{invariant Ruelle distribution $\mathcal{T}_{\lambda_0}$} associated with the resonance $\lambda_0$ (see e.g. \cite{Zelditch.2007}):
\begin{equation*}
\mathrm{Res}_{\lambda = \lambda_0} \left[ Z_f^\mathbf{X}(\lambda) \right] = \mathrm{tr}^\flat \left( \Pi_{\lambda_0} f \right) \defgl \mathcal{T}_{\lambda_0}(f) .
\end{equation*}
Using a result derived by Guillarmou, Hilgert, and Weich \cite{Weich.2021} that connects the invariant Ruelle distributions with so-called \emph{Patterson-Sullivan distributions} $\mathrm{PS}_{\varphi_l}$, we are able to generalize the description of $\mathrm{PS}_{\varphi_l}$ as residues previously derived by Anantharaman and Zelditch \cite{Zelditch.2007} to a broader class of systems and weights. For the precise statement see \Cref{thm:ps_residue}.

\subsection{Paper organization}

We begin with a detailed introduction to our geometric setting in Section \ref{setup}. This entails first the notion of open hyperbolic systems as used in \cite{Dyatlov.2016a} (Section \ref{setup.1}) and second a setup quite similar to open hyperbolic systems but without the requirement of strict convexity (Section \ref{setup.2}). In particular we define the restricted resolvent, Pollicott-Ruelle resonances and invariant Ruelle distributions. For our numerics it turns out to be useful to be able to restrict the invariant Ruelle distributions to certain hypersurfaces transversal to the flow (Section \ref{setup.3}).

The following sections are dedicated to weighted zeta functions and their meromorphic continuation: Section \ref{trace} gives a proof of a weighted version of Guillemin's trace formula, while Section \ref{proof} introduces weighted zeta functions and proves their meromorphic continuation using the trace formula and wavefront estimates for the restricted resolvent. Both sections closely follow the approaches of \cite{Dyatlov.2016} and \cite{Dyatlov.2016a}.

In Sections \ref{setup} -- \ref{proof} we worked with the setting of general hyperbolic flows with compact trapped set. In Section \ref{ps} we restrict to the special case of geodesic flows on compact locally symmetric spaces of rank one. For these systems our residue formula \eqref{eq:laurent} can be interpreted as a classical description of quantum mechanical phase space distributions known as \emph{Patterson-Sullivan distributions}. This allows us to extend results of Anantharaman and Zelditch \cite{Zelditch.2007} as well as Emonds \cite{Emonds.2014} to general rank one, compact, locally symmetric spaces and arbitrary weight functions $f\in \mathrm{C}^\infty(\mathcal{M})$.

One of our main motivations for proving \Cref{thm:analyt_zeta} is the fact that dynamical zeta functions can often be calculated numerically and in a very efficient manner. Our theorem then immediately allows us to calculate the invariant Ruelle distributions numerically. So far these distributions have not been studied in depth and we hope that they encode certain aspects of the spectrum similar to the Laplace eigenfunctions in quantum chaos. As a first proof-of-principle we show in Appendix \ref{numerics} some exemplary numerical calculations of invariant Ruelle distributions for the $3$-disc scattering systems. Even in the few plots provided here one can already observe some interesting behavior which could serve as a starting point for more in depth numerical investigations.

\subsection*{Acknowledgments}

This work has received funding from the Deutsche Forschungsgemeinschaft (DFG) (Grant No. WE 6173/1-1 Emmy Noether group “Microlocal Methods for Hyperbolic Dynamics”) and by an individual grant of P.S. from the Studienstiftung des Deutschen Volkes. We acknowledge helpful discussions concerning the relation to the theory of Patterson-Sullivan distributions with Joachim Hilgert. P.S. acknowledges numerous fruitful discussions with Semyon Dyatlov during a research stay at MIT in spring 2020.

	
\section{Geometric Setup} \label{setup}
	
	First, we introduce the notion of \textit{open hyperbolic system} as it was used in \cite{Dyatlov.2016a} in Section \ref{setup.1}. This setting will allow us to employ the meromorphically continued restricted resolvent which makes the upcoming proof of our main theorem in Section \ref{proof.1} rather straight forward later on. Afterwards in Section \ref{setup.2} we extend these results slightly by removing the requirement of strict convexity following ideas developed in \cite{Guillarmou.2021}. This will allow us to give a more practically feasible formulation of our main result in Section \ref{proof.2}. The underlying idea was already used by the authors in their paper \cite{schuette.2021} together with Benjamin Küster in the slightly more concrete setting of convex obstacle scattering. We finish this section with the proof of a dimensional reduction technique for invariant Ruelle distributions which is essential for the numerical calculations in Appendix \ref{numerics}.

	\subsection{Open Hyperbolic Systems} \label{setup.1}
	
	Let $\overline{\mathcal{U}}$ be a compact, $n$-dimensional smooth manifold with interior $\mathcal{U}$ and smooth boundary $\partial\mathcal{U}$. We may have $\partial\mathcal{U} = \emptyset$, i.e. the case of a closed manifold. The dynamical object of main interest is a non-vanishing $\mathrm{C}^\infty$-vector field $X$ on $\overline{\mathcal{U}}$. We denote by $\varphi_t = \varphi_t^X = \mathrm{e}^{tX}$ the corresponding flow on $\overline{\mathcal{U}}$, which for $\partial\mathcal{U} \neq \emptyset$ will generally not be complete.
	
	Before stating the dynamical requirements for $X$ we have to introduce some auxiliary objects. Concretely, let $\rho\in\mathrm{C}^\infty(\overline{\mathcal{U}}, [0, \infty[)$ be a boundary defining function, that is $\rho^{-1}(0) = \partial\mathcal{U}$ and $\mathrm{d}\rho(x) \neq 0$ for any $x\in\partial\mathcal{U}$. Boundary defining functions always exist (\cite[p.~118]{Lee.2012}). Furthermore, let $\mathcal{M}\supseteq\overline{\mathcal{U}}$ be an embedding into a \textit{compact manifold without boundary}, such that $X$ extends onto $\mathcal{M}$ in a manner making $\mathcal{U}$ \textit{convex}\footnote{We keep the notation $X$ and $\varphi_t$ for the continuations of our vector field and its flow; note that the continued flow is now complete by compactness of $\mathcal{M}$ so the upcoming definitions make sense.}:
	\begin{equation*}
	x,~ \varphi_T(x)\in \mathcal{U} ~\text{for some}~ T > 0 \quad\Longrightarrow\quad \varphi_t(x)\in \mathcal{U} ~\forall t\in [0, T] ~.
	\end{equation*}
	Such an ambient manifold $\mathcal{M}$ and continuation of $X$ always exist by \cite[Lemma~1.1]{Dyatlov.2016a}, given that (A1) below holds. Then \cite[Lemma~1.2]{Dyatlov.2016a} asserts that
	\begin{equation*}
	\Gamma_\pm \defgr \bigcap_{\pm t \geq 0} \varphi_t(\overline{\mathcal{U}}), \quad K\defgr \Gamma_+ \cap \Gamma_- \subseteq \mathcal{U} ~.
	\end{equation*}
	The set $K$, called the \textit{trapped set} of $X$, is independent of $\mathcal{M}$ and the extension of $X$. Finally, let $\mathcal{E}$ be a complex $\mathrm{C}^\infty$-vector bundle over $\overline{\mathcal{U}}$ and $\mathbf{X}: \mathrm{C}^\infty(\overline{\mathcal{U}}, \mathcal{E}) \rightarrow \mathrm{C}^\infty(\overline{\mathcal{U}}, \mathcal{E})$ a first order differential operator.
	
	With these preliminaries given we call the triple $(\mathcal{U}, \varphi_t, \mathbf{X})$ an open hyperbolic system if it satisfies the following dynamical requirements:
	\begin{itemize}
		\item[(A1)] The boundary $\partial \mathcal{U}$ is \textit{strictly convex}, i.e.:
		\begin{equation*}
		x\in\partial\mathcal{U},~ (X\rho)(x) = 0 \quad \Longrightarrow \quad X(X\rho)(x) < 0 ~.
		\end{equation*}
		This condition is independent of the choice of $\rho$.
		\item[(A2)] The flow $\varphi_t$ is \textit{hyperbolic on $K$}, i.e. for any $x\in K$ the tangent space splits as
		\begin{equation*}
		T_x\mathcal{M} = \mathbb{R}\cdot X(x) \oplus E_s(x) \oplus E_u(x) ~,
		\end{equation*}
		where $E_s$ and $E_u$ are continuous in $x\in K$, invariant under $\varphi_t$ and there exist constants $C_0, C_1 > 0$ with
		\begin{equation} \label{eq_def_hyperbolicity}
		\begin{split}
		\Arrowvert \mathrm{d}\varphi_t(x) v\Arrowvert_{T_{\varphi_t(x)}\mathcal{M}} &\leq C_0 \exp(-C_1 t) \Arrowvert v\Arrowvert_{T_x\mathcal{M}}, \qquad t\geq 0,~ v\in E_s(x) \\
		\Arrowvert \mathrm{d}\varphi_t(x) v\Arrowvert_{T_{\varphi_t(x)}\mathcal{M}} &\geq C_0^{-1} \exp(C_1 t) \Arrowvert v\Arrowvert_{T_x\mathcal{M}}, \qquad t\geq 0,~ v\in E_u(x) ~.
		\end{split}
		\end{equation}
		$\Arrowvert\cdot \Arrowvert$ denotes any continuous norm on the tangent bundle.\footnote{The definition does not depend on the specific norm as $K$ is compact. The same holds for upcoming arbitrary choices of densities, etc.}
		\item[(A3)] $\mathbf{X}$ satisfies
		\begin{equation} \label{eq_leibniz}
		\mathbf{X}(f\mathbf{u}) = (Xf) \mathbf{u} + f (\mathbf{X}\mathbf{u}),\qquad f\in\mathrm{C}^\infty(\overline{\mathcal{U}}),~ \mathbf{u}\in \mathrm{C}^\infty(\overline{\mathcal{U}}, \mathcal{E}) ~.
		\end{equation}
		We also denote by $\mathbf{X}$ an arbitrary extension onto $\mathcal{M}$ that still satisfies \eqref{eq_leibniz}.
	\end{itemize}
	
	Within the setting just described we will need some additional dynamical objects derived from the flow $\varphi_t$ and the operator $\mathbf{X}$. We start by fixing a smooth density on $\mathcal{M}$ and a smooth scalar product on $\mathcal{E}$. That lets us consider the \textit{transfer operator} $\exp(-t \mathbf{X}): \mathrm{L}^2(\mathcal{M}, \mathcal{E}) \rightarrow \mathrm{L}^2(\mathcal{M}, \mathcal{E})$ associated with $\mathbf{X}$, i.e. the solution semigroup of
	\begin{equation*}
	\frac{\partial}{\partial t}\mathbf{v}(t, x) = \left(-\mathbf{X} \mathbf{v}\right)(t, x),\qquad \mathbf{v}(0, \cdot) = \mathbf{u} ~.
	\end{equation*}
	This allows us to define the linear \textit{parallel transport} map at $x\in\mathcal{U}$ and $t > 0$ such that $\varphi_t(x)\in \mathcal{U}$:
	\begin{equation*}
	\begin{split}
	\alpha_{x, t}: \mathcal{E}_x &\longrightarrow \mathcal{E}_{\varphi_t(x)} \\
	u &\longmapsto \left(\mathrm{e}^{-t\mathbf{X}} \mathbf{u} \right) (\varphi_t(x)) ~,
	\end{split}
	\end{equation*}
	where $\mathbf{u}$ is some smooth section of $\mathcal{E}$ with $\mathbf{u}(x) = u$. The definition does not depend on the choice of $\mathbf{u}$ because $\mathbf{u}(x) = 0$ implies $\left(\mathrm{e}^{-t\mathbf{X}} \mathbf{u} \right) (\varphi_t(x)) = 0$ by \cite[Eq.~(0.8)]{Dyatlov.2016a} and an expansion in terms of a local frame for $\mathcal{E}$. In particular, if $\gamma(t) = \varphi_t(x_0)$ is a closed orbit with period $T_\gamma$ then $\alpha_{x_0, T_\gamma} = \alpha_{x_0, t}^{-1} \circ \alpha_{\gamma(t), T_\gamma}\circ \alpha_{x_0, t}$ and the trace
	\begin{equation*}
	\mathrm{tr}(\alpha_\gamma) \defgr \mathrm{tr}(\alpha_{\gamma(t), T_\gamma})
	\end{equation*}
	is well-defined independent of $t$.
	
	Next we define the so-called \textit{linearized Poincar\'{e} map} which at $x\in\mathcal{U}$ and $t > 0$ satisfying $\varphi_t(x)\in \mathcal{U}$ is defined as the linear map
	\begin{equation*}
	\mathcal{P}_{x, t} \defgr \mathrm{d}\varphi_{-t}(x)\restr{E_s(x)\oplus E_u(x)} ~.
	\end{equation*}
	Again we will use this map if $x$ is located on a closed trajectory $\gamma$. In this case $\mathcal{P}_{x, T_\gamma}$ is an endomorphism and $\mathrm{det}(\mathrm{id} - \mathcal{P}_\gamma) \defgr \mathrm{det}(\mathrm{id} - \mathcal{P}_{x, T_\gamma})$ is independent of $x$ because similarly as for the parallel transport one finds that $\mathcal{P}_{x, T_\gamma}$ is conjugate to any $\mathcal{P}_{y, T_\gamma}$ provided $y$ also belongs to $\gamma$.
	
	We finish this section by recasting the definition of \emph{invariant Ruelle distributions} from \cite{Weich.2021} into our setting: Let $(\overline{\mathcal{U}}, \varphi_t, \mathbf{X})$ be an open hyperbolic system, then we denote by $\mathbf{R}(\lambda) \defgr \mathbf{1}_\mathcal{U} (\mathbf{X} + \lambda)^{-1} \mathbf{1}_\mathcal{U}: \mathrm{C}^\infty_\mathrm{c}(\mathcal{U}, \mathcal{E}) \rightarrow \mathcal{D}'(\mathcal{U}, \mathcal{E})$ $\lambda_0$ its \textit{restricted resolvent}. In \cite{Dyatlov.2016a} it was proven that $\mathbf{R}(\lambda)$ continues meromorphically onto $\mathbb{C}$ and its poles are called \emph{Ruelle resonances}. Furthermore, the residue of $\mathbf{R}(\lambda)$ at a resonance $\lambda_0$ is a finite rank operator $\Pi_{\lambda_0}: \mathrm{C}^\infty_\mathrm{c}(\mathcal{U}, \mathcal{E}) \rightarrow \mathcal{D}'(\mathcal{U}, \mathcal{E})$ that satisfies \cite[Thm.~2]{Dyatlov.2016a}
	\begin{equation} \label{eq_pi_commutator}
	\mathbf{X} \Pi_{\lambda_0} = \Pi_{\lambda_0} \mathbf{X}, \quad \mathrm{supp}\left(K_{\Pi_{\lambda_0}} \right) \subseteq \Gamma_+\times \Gamma_-, \quad \mathrm{WF}'\left( \Pi_{\lambda_0} \right) \subseteq E_+^* \times E_-^* ,
	\end{equation}
	where $K_{\Pi_{\lambda_0}}$ is the Schwartz kernel of $\Pi_{\lambda_0}$, $\mathrm{WF}'$ denotes the wavefront set, and $E_\pm^*$ are subbundles of $T^*\mathcal{U}$ over $\Gamma_\mp$ constructed in \cite[Lemma~1.10]{Dyatlov.2016a}. In particular, the support property in \eqref{eq_pi_commutator} shows that the restriction $K_{\Pi_{\lambda_0}} \big|_{\Delta}$ to the diagonal $\Delta\subseteq \mathcal{U}\times \mathcal{U}$ is supported in $\Gamma_+\cap \Gamma_- = K$ and therefore yields an element of $\mathcal{E}'(\mathcal{U})$. Now combining this with the wavefront estimate \cite[Eq.~(0.14)]{Dyatlov.2016a} of $K_{\Pi_{\lambda_0}}$ guarantees that the following map is well-defined:
	\begin{equation} \label{def_invariant_ruelle}
	\mathcal{T}_{\lambda_0}: 
	\begin{cases}
	\mathrm{C}^\infty(\mathcal{U}) &\longrightarrow ~\qquad \mathbb{C}\\
	\quad f\quad &\longmapsto ~\mathrm{tr}^\flat\left( \Pi_{\lambda_0} f \right)
	\end{cases} .
	\end{equation}
	The generalized density $\mathcal{T}_{\lambda_0}$ is called the \emph{invariant Ruelle distribution} associated with $\lambda_0$. Note that we can re-write $\mathcal{T}_{\lambda_0}(f) = \mathrm{tr}^\flat( \Pi_{\lambda_0} f ) = \mathrm{tr}^\flat( f \Pi_{\lambda_0} )$ as the trace is cyclic.
	
	The adjective \emph{invariant} is justified because $\mathcal{T}_{\lambda_0}$ is indeed invariant under the flow: The Leibniz rule \eqref{eq_leibniz} implies $\Pi_{\lambda_0} (Xf) = \Pi_{\lambda_0} \mathbf{X} f - \Pi_{\lambda_0} f \mathbf{X}$ which together with the vanishing of the commutator \eqref{eq_pi_commutator} and the cyclic property yields
	\begin{equation*}
	\begin{split}
	\mathcal{T}_{\lambda_0}(Xf) &= \mathrm{tr}^\flat\left( \Pi_{\lambda_0} \mathbf{X} f - \Pi_{\lambda_0} f \mathbf{X} \right)\\
	&= \mathrm{tr}^\flat\left( \mathbf{X} \Pi_{\lambda_0} f - \mathbf{X}\Pi_{\lambda_0} f \right) = 0 .
	\end{split}
	\end{equation*}

	\subsection{Removing Strict Convexity for Resolvents} \label{setup.2}
	
	In practical applications it often turns out to be difficult to verify the strict convexity condition. To circumvent this difficulty we recast the meromorphic continuation achieved in \cite{Dyatlov.2016a} into a simpler setting by constructing a perturbation of the generator $X$ of $\varphi_t$ and of the operator $\mathbf{X}$. Our proof mostly follows \cite{Guillarmou.2021} which uses techniques developed in \cite{Conley.1971} and \cite{Robinson.1980}. For a very similar application but without the extension to the vector valued case see \cite{Dyatlov.2018}.
	
	Let $\mathcal{M}$ be a smooth manifold with manifold interior $\mathring{\mathcal{M}}$ and smooth, possibly empty boundary, $X$ a smooth, nowhere vanishing vector field on $\mathcal{M}$, and $\varphi_t$ the flow associated with $X$. Furthermore, let a smooth, complex vector bundle $\mathcal{E}$ over $\mathcal{M}$ together with a first-order differential operator $\mathbf{X}: \mathrm{C}^\infty(\mathcal{M}, \mathcal{E}) \rightarrow \mathrm{C}^\infty(\mathcal{M}, \mathcal{E})$ be given. We assume that $\mathbf{X}$ satisfies the Leibniz rule \eqref{eq_leibniz}, that the trapped set $K$ of $\varphi_t$ defined in \eqref{trapped_set_def} is compact and satisfies $K\subseteq \mathring{\mathcal{M}}$, and finally that $\varphi_t$ is hyperbolic on $K$ as defined in \eqref{eq_def_hyperbolicity}.
	
	First we observe that we may assume $\mathcal{M}$ to be compact with smooth, non-empty boundary: If it is not then there exists a compact submanifold with smooth boundary of $\mathcal{M}$ which still contains $K$ in its manifold interior. This follows by standard tools in smooth manifold theory, namely via a smooth exhaustion function (\cite[Prop.~2.28]{Lee.2012}) combined with Sard's theorem. If necessary, we may replace $\mathcal{M}$ with this submanifold.
	
	Now \cite[Prop.~2.2 and Lemma 2.3]{Guillarmou.2021} immediately yields the existence of a compact $\overline{\mathcal{U}}_0\subseteq \mathcal{M}$ with manifold interior $\mathcal{U}_0$ and smooth boundary $\partial\mathcal{U}_0$ such that $K\subseteq \mathcal{U}_0$, and the existence of a smooth vector field $X_0$ on $\mathcal{M}$ such that $\partial \mathcal{U}_0$ is strictly convex w.r.t. $X_0$ in the above sense and $X - X_0$ is supported in an arbitrarily small neighbourhood of $\partial \mathcal{U}_0$. In addition, the trapped set of $X_0$ coincides with $K$. Finally, we may assume both the flow $\varphi_t$ as well as the flow $\varphi_t^0$ of $X_0$ to be complete by embedding $\mathcal{M}$ into an ambient closed manifold and extending $X, X_0$ arbitrarily. In this setting we define the \emph{escape times} from some compact set $A\subseteq \mathcal{M}$ to be
	\begin{equation*}
	\begin{split}
	\tau^\pm_A(x) &\defgr \pm \sup \left\{ t\geq 0 \,\big|\, \varphi_{\pm s}(x)\in A\, \forall s\in [0, t] \right\} ,\\
	\tau^{0, \pm}_A(x) &\defgr \pm \sup \left\{ t\geq 0 \,\big|\, \varphi^0_{\pm s}(x)\in A\, \forall s\in [0, t] \right\} , \qquad x\in A .
	\end{split}
	\end{equation*}
	The respective \emph{forward and backward trapped sets} are then given by
	\begin{equation*}
	\Gamma_\pm(A) \defgr \left\{x\in A \,\big|\, \tau^\mp_A(x) = \mp\infty \right\}, \quad \Gamma_\pm^0(A) \defgr \left\{x\in A \,\big|\, \tau^{0, \mp}_A(x) = \mp\infty \right\} ,
	\end{equation*}
	such that in particular we have $K = \Gamma_+(\overline{\mathcal{U}}_0)\cap \Gamma_-(\overline{\mathcal{U}}_0)$.
	
	Next we require an appropriate perturbation of $\mathbf{X}$ which satisfies the Leibniz rule \eqref{eq_leibniz} with respect to $X_0$ instead of $X$. But this is straight forward: Given any section $\mathbf{u}\in \mathrm{C}^\infty(\overline{\mathcal{U}}, \mathcal{E})$ and $x\in\overline{\mathcal{U}}$ we consider a local frame $\mathbf{e}_i$ in a neighbourhood of $x$. Then $\mathbf{u}$ expands as $\mathbf{u} = \sum_i u^i \mathbf{e}_i$ in this neighbourhood and we define
	\begin{equation*}
	\mathbf{X}_0 \mathbf{u} \defgr \mathbf{X}\mathbf{u} + \sum_i \left( \left( X_0 - X \right) u^i\right) \cdot\mathbf{e}_i .
	\end{equation*}
	The second term on the right-hand side obviously yields a well-defined first-order differential operator and $\mathbf{X} - \mathbf{X}_0$ is supported near $\partial\mathcal{U}_0$. Given $f\in \mathrm{C}^\infty(\overline{\mathcal{U}})$ we can verify the Leibniz rule via the following calculation in a neighbourhood of $x$:
	\begin{equation*}
	\begin{split}
	\mathbf{X}_0 (f \cdot\mathbf{u}) &= \mathbf{X}(f \cdot\mathbf{u}) + \sum_i \left( \left( X_0 - X \right) f u^i \right) \cdot\mathbf{e}_i\\
	&= (X f) \cdot\mathbf{u} + f \cdot\left( \mathbf{X} \mathbf{u} \right) + \left( X_0 f - X f \right) \cdot\mathbf{u} + f\cdot \sum_i \left( \left(X_0 - X\right) u^i \right) \cdot\mathbf{e}_i\\
	&= f\cdot \left( \mathbf{X}_0 \mathbf{u} \right) + \left( X_0 f \right) \mathbf{u} .
	\end{split}
	\end{equation*}
	
	We are now in a position to state and prove the main result of this section. For some open set $\mathcal{O}\subseteq \mathcal{M}$, a section $\mathbf{u}\in \mathrm{C}^\infty_\mathrm{c}(\mathcal{O}, \mathcal{E})$, and a spectral parameter $\lambda\in\mathbb{C}$ consider the \textit{resolvent} given by the following formal integral:
	\begin{equation} \label{def_gen_resolvent}
	\mathbf{R}_\mathcal{O}(\lambda) \mathbf{u}(x) \defgr \int_0^{-\tau^-_{\overline{\mathcal{O}}}(x)} \mathrm{e}^{-(\mathbf{X} + \lambda) t} \mathbf{u}(x) \mathrm{d}t ,
	\end{equation}
	where $\exp(-\mathbf{X} t)$ denotes the transfer operator associated with $\mathbf{X}$ as defined in the previous section. At this point, the expression in \eqref{def_gen_resolvent} remains formal because on the one hand $\tau^-$ may be of low regularity and on the other hand the integral may not converge for $\tau^-(x) = -\infty$. By choosing an appropriate $\mathcal{O}$ we can get around these issues and obtain the following theorem:
	
	\begin{theorem}{Meromorphic Resolvent Without Strict Convexity}{resolvent_without_strict_convexity}
	Let $\mathcal{M}$ be a smooth manifold with smooth boundary, $X$ a smooth, nowhere vanishing vector field on $\mathcal{M}$, and $K$ the trapped set of the flow $\varphi_t$ of $X$. Also, let $\mathcal{E}$ be a smooth, complex vector bundle over $\mathcal{M}$ and $\mathbf{X}$ a first-order differential operator on $\mathcal{E}$.
	
	Assume that $\mathbf{X}$ satisfies the Leibniz rule \eqref{eq_leibniz}, $K$ is compact, $K\subseteq \mathring{\mathcal{M}}$, and $\varphi_t$ is hyperbolic on $K$. Then there exists an arbitrarily small compact $\overline{\mathcal{U}}\subseteq \mathcal{M}$ which contains $K$ in its interior $\mathcal{U}$ such that
	\begin{enumerate}
		\item $\mathbf{R}_\mathcal{U}(\lambda)$ is well-defined for $\mathrm{Re}(\lambda) \gg 0$ as an operator on $\mathrm{L}^2(\mathcal{M}, \mathcal{E})$ and satisfies $(\mathbf{X} + \lambda) \mathbf{R}_\mathcal{U}(\lambda) = \mathrm{id}_{\mathrm{C}^\infty_\mathrm{c}(\mathcal{U}, \mathcal{E})}$,
		\item $\mathbf{R}_\mathcal{U}(\lambda)$ continues meromorphically to $\mathbb{C}$ as a family of operators $\mathrm{C}^\infty_\mathrm{c}(\mathcal{U}, \mathcal{E}) \rightarrow \mathcal{D}'(\mathcal{U}, \mathcal{E})$ with poles of finite rank,
		\item the residue $\Pi_{\lambda_0}$ of $\mathbf{R}_\mathcal{U}(\lambda)$ satisfies \eqref{eq_pi_commutator}.
	\end{enumerate}
	\end{theorem}

	\begin{proof}
	For this proof we assume the setting and notation introduced in this section. Observe that both $X = X_0$ and $\mathbf{X} = \mathbf{X}_0$ on any compact $\overline{\mathcal{U}}$ with $\overline{\mathcal{U}} \subseteq \mathcal{U}_0$. If we additionally had for all $t\geq 0$ the property
	\begin{equation} \label{eq_dyn_convex}
	x, \varphi^0_t(x)\in \overline{\mathcal{U}} \quad \Longrightarrow\quad \varphi^0_s(x)\in \overline{\mathcal{U}} \,\forall s\in [0, t] ,
	\end{equation}
	then it would immediately follow for any $x\in \mathcal{U}$, $\mathbf{u}\in \mathrm{C}^\infty_\mathrm{c}(\mathcal{U}, \mathcal{E})$ and $\mathrm{Re}(\lambda) \gg 0$ that
	\begin{equation*}
	\begin{split}
	\mathbf{R}_\mathcal{U}(\lambda) \mathbf{u}(x) &= \int_0^{-\tau^-_{\overline{\mathcal{U}}}(x)} \mathrm{e}^{-(\mathbf{X} + \lambda) t} \mathbf{u}(x) \mathrm{d}t\\ &= \int_0^\infty \mathrm{e}^{-(\mathbf{X}_0 + \lambda) t} \mathbf{u}(x) \mathrm{d}t = \mathbf{1}_\mathcal{U} \left( \mathbf{X}_0 + \lambda \right)^{-1} \mathbf{1}_\mathcal{U} \mathbf{u}(x),
	\end{split}
	\end{equation*}
	and the claims (1) - (3) would be shown by a straight forward application of the material from \cite{Dyatlov.2016a} recalled in Section \ref{setup.1}.
	
	It therefore remains to construct an arbitrarily small $\overline{\mathcal{U}}$ satisfying \eqref{eq_dyn_convex}. First note that we may assume \eqref{eq_dyn_convex} to hold on $\overline{\mathcal{U}}_0$ by choosing the extension of $X_0$ appropriately, see \cite[Eq.~(0.2)]{Dyatlov.2016a}. Given any open $\mathcal{O}$ with $K\subseteq \mathcal{O} \subseteq \overline{\mathcal{O}} \subseteq \mathcal{U}_0$ we can then set
	\begin{equation*}
	\overline{\mathcal{U}} \defgr \varphi^0_{-T}(\overline{\mathcal{U}}_0) \cap \overline{\mathcal{U}}_0 \cap \varphi^0_T(\overline{\mathcal{U}}_0) ,
	\end{equation*}
	where choosing $T > 0$ large enough guarantees $\overline{\mathcal{U}}\subseteq \mathcal{O}$ by \cite[Lemma~1.4]{Dyatlov.2016a}. It is now easily verified that $\overline{\mathcal{U}}$ satisfies \eqref{eq_dyn_convex}.
	\end{proof}
	In Section \ref{proof} we will show that the poles of $\mathbf{R}_\mathcal{U}(\lambda)$ are independent of the choice of a set $\overline{\mathcal{U}}$ (and also independent of $\overline{\mathcal{U}}_0$). They are called the \emph{Pollicott-Ruelle resonances} of $\varphi_t$ and will figure prominently below as the poles of our weighted zeta function. Furthermore we can define invariant Ruelle distributions $\mathcal{T}_{\lambda_0}$ in the setting of \Cref{thm:resolvent_without_strict_convexity} by the same formula \eqref{def_invariant_ruelle} presented in the previous Section \ref{setup.1}.
	
	\begin{remark}
	The residue $\Pi_{\lambda_0}$ at a resonance $\lambda_0$ is also independent of the set $\overline{\mathcal{U}}$ in the sense that
	\begin{equation*}
	\Pi_{\lambda_0}\big|_{\mathrm{C}^\infty_\mathrm{c}(\mathcal{U}', \mathcal{E})} = \Pi_{\lambda_0}' ,
	\end{equation*}
	where $\overline{\mathcal{U}}'\subseteq \mathcal{U}$ denotes a second set on which the resolvent can be continued meromorphically and $\Pi_{\lambda_0}'$ denotes the residue of this continuation $\mathbf{R}_{\mathcal{U}'}(\lambda)$. This follows immediately by uniqueness of meromorphic continuation and the fact that the restriction of $\mathbf{R}_\mathcal{U}(\lambda)$ to $\mathrm{C}^\infty_\mathrm{c}(\mathcal{U}', \mathcal{E})$ coincides with $\mathbf{R}_{\mathcal{U}'}(\lambda)$ for $\mathrm{Re}(\lambda) \gg 0$. In particular, note that this implies the independence of $\mathcal{T}_{\lambda_0}$ from $\mathcal{U}$.
	
	An analogous independence statement holds for $\overline{\mathcal{U}}_0$.
	\end{remark}

	\subsection{Restricting Ruelle Distributions} \label{setup.3}
	
	One of the main applications of our weighted zeta functions $Z_f$ is the concrete numerical calculation of invariant Ruelle distributions $\mathcal{T}_{\lambda_0}$ for certain $3$-dimensional dynamical systems. For ease of calculation as well as plotting we would like to reduce this to a reasonable distribution on a $2$-dimensional space. In this section we present a general theorem which allows just this.

	For the general setup let $\mathcal{T}_{\lambda_0}\in \mathcal{D}'(\mathcal{M})$ be an invariant Ruelle distribution in the setting of \Cref{thm:resolvent_without_strict_convexity}. We will call a smooth submanifold $\Sigma\subseteq \mathcal{M}$ a \emph{Poincar\'{e} section for $\varphi_t$} if $\Sigma$ has codimension one and is transversal to $\varphi_t$ on the trapped set $K$, i.e. for each $x\in K$ we have
	\begin{equation*}
	T_x\mathcal{M} = \mathbb{R}\cdot X(x)\oplus T_x \Sigma ,
	\end{equation*}
	where as above $X$ denotes the generator of $\varphi_t$. With this definition at hand we can state and prove the following lemma:

\begin{lemma}{Restriction of Ruelle Distributions}{thm_ruelle_restriction}
	Let $\mathcal{T}_{\lambda_0}$ be an invariant Ruelle distribution for a dynamical system $(\mathcal{M}, \varphi_t, \mathbf{X})$ as in \Cref{thm:resolvent_without_strict_convexity} and $\Sigma\subseteq \mathcal{M}$ a Poincar\'{e} section for $\varphi_t$. Then the pullback of $\mathcal{T}_{\lambda_0}$ along $\iota_\Sigma: \Sigma\hookrightarrow \mathcal{U}$ is well-defined and will be called the \emph{restriction to $\Sigma$}:
	\begin{equation}
	\mathcal{T}_{\lambda_0}\big|_{\Sigma} \defgr \left( \iota_\Sigma \right)^* \mathcal{T}_{\lambda_0} \in \mathcal{D}'(\Sigma) .
	\end{equation}
\end{lemma}

\begin{proof}
	For the existence of the pullback we use the classical Hörmander condition \cite[Theorem~8.2.4]{Hoerm.2008} similar to the proof of \Cref{thm:analyt_zeta}, i.e. what we need to show is
	\begin{equation*}
	\mathrm{WF}\left( \mathcal{T}_{\lambda_0} \right) \cap N^* \Gamma_{\iota_\Sigma} = \{ 0 \} ,
	\end{equation*}
	with $N^* \Gamma_{\iota_\Sigma}$ the conormal bundle to the graph of $\iota_\Sigma$. Now the invariance property $X \mathcal{T}_{\lambda_0} = 0 \in \mathrm{C}^\infty(\mathcal{M})$ proven above immediately implies the following estimate on $\mathrm{WF}\left( \mathcal{T}_{\lambda_0} \right)$:	
	\begin{equation*}
	\mathrm{WF}\left( \mathcal{T}_{\lambda_0} \right) \subseteq E_s^*\oplus E_u^* \subseteq T^* \mathcal{M} .
	\end{equation*}
	Now the conormal bundle is explicitly given by
	\begin{equation*}
	N^* \Gamma_{\iota_\Sigma} = \left\{ (x, \xi) \,\bigg|\, x\in\Sigma,\, \xi\big|_{T_x \Sigma} = 0 \right\} \subseteq T^*\mathcal{M},
	\end{equation*}
	and an element $(x, \xi)\in \mathrm{WF}( \mathcal{T}_{\lambda_0} ) \cap N^* \Gamma_{\iota_\Sigma}$ therefore satisfies $x\in K\cap \Sigma$, $\xi(T_x \Sigma) = 0$ and $\xi(X_x) = 0$ by the definition of $E_s^*$ and $E_u^*$. Now the transversality condition $T_x \mathcal{M} = \mathbb{R}\cdot X(x) \oplus T_x \Sigma$ yields $\xi = 0$.
\end{proof}

	
\section{A Weighted Trace Formula} \label{trace}

Our main tool for connecting the restricted resolvent with our weighted zeta function is a weighted version of the Atiyah-Bott-Guillemin trace formula which we will present in this section. The proof is done in three steps: First we show that the left-hand side of the trace formula is well-defined by estimating its wavefront set (Section \ref{trace.1}). Next we prove a local version of the trace formula which amounts to choosing suitable coordinates and manipulating the resulting distributions on $\mathbb{R}^n$ (Section \ref{trace.2}). Finally we combine the local trace formula with a partition of unity argument to obtain the global version (Section \ref{trace.3}). An unweighted formulation for open systems can be found in \cite[Eq.~(4.6)]{Dyatlov.2016a} and an analogous result for the compact case is presented in \cite[Eq.~(2.4)]{Dyatlov.2016}.

\subsection{Statement and Wavefront Estimate} \label{trace.1}

Our overall goal in the following sections is to prove the weighted Atiyah-Bott-Guillemin trace formula:

\begin{lemma}{Weighted Atiyah-Bott-Guillemin Trace Formula}{guillemin_trace}
	For any cut-offs $\chi\in \mathrm{C}^\infty_\mathrm{c}(\mathbb{R}\backslash \{0\})$ and $\widetilde{\chi}\in \mathrm{C}^\infty_\mathrm{c}(\mathcal{U})$, with $\widetilde{\chi} \equiv 1$ near the trapped set $K$, the following holds:
	\begin{equation} \label{eq_open_systems3}
	\mathrm{tr}^\flat\left( \int_\R \chi(t) \widetilde{\chi} \mathrm{e}^{-t\mathbf{X}} f \widetilde{\chi} \mathrm{d}t \right)
	=
	\sum_\gamma \frac{\chi(T_\gamma) \mathrm{tr}(\alpha_\gamma)}{\vert \det(\mathrm{id} - \mathcal{P}_\gamma) \vert} \int_{\gamma^\#} f ~,
	\end{equation}
	where the sum is over all closed orbits $\gamma$ of $\varphi_t$.
\end{lemma}

\begin{proof}
	The proof is directly adapted from \cite[Sec.~4.1]{Dyatlov.2016a} and \cite[App.~B]{Dyatlov.2016} but the main points can already be found in \cite[§2 of Lecture~2]{Guillemin.1977}. We proceed in three steps spread out over this and the following two sections:

	\vspace{0.25cm}
	\begin{enumerate}
		\item[\textbf{(1.)}] Show that the flat trace on the left-hand side is well-defined (Section \ref{trace.1}).
		\item[\textbf{(2.)}] Prove a local version of the theorem (Section \ref{trace.2}).
		\item[\textbf{(3.)}] Combine the local version with a partition of unity argument to prove the global theorem (Section \ref{trace.3}).
	\end{enumerate}
	\vspace{0.25cm}
	
	First of all, let $\chi\in \mathrm{C}^\infty_\mathrm{c}(\mathbb{R}\backslash \{0\})$ be given. Then $\mathbf{A}_{f, \chi} \defgr \int_\mathbb{R} \chi(t) \widetilde{\chi}\mathrm{e}^{-t\mathbf{X}} f \widetilde{\chi} \mathrm{d}t$ is an operator
	\begin{equation*}
	\mathbf{A}_{f, \chi}: \mathrm{C}^\infty(\mathcal{M}, \mathcal{E}) \longrightarrow \mathrm{C}^\infty(\mathcal{M}, \mathcal{E})\subseteq \mathcal{D}'(\mathcal{M}, \mathcal{E}) ~,
	\end{equation*}
	and via the Schwartz kernel theorem we can consider the integrand $\widetilde{\chi} \mathrm{e}^{-t\mathbf{X}} f \widetilde{\chi}$ as an operator\footnote{Given bundles $\mathcal{E}_i\rightarrow M_i$ then $\mathcal{E}_1\boxtimes\mathcal{E}_2$ denotes the tensor product of the pullbacks of the $\mathcal{E}_i$ onto $M_1\times M_2$. Note that \cite{Dyatlov.2016a} uses the notation $\mathrm{End}(\mathcal{E})$ for the bundle $\mathcal{E}\boxtimes\mathcal{E}^*$ over $\mathcal{M} \times \mathcal{M}$.}
	\begin{equation*}
	\mathrm{C}^\infty_\mathrm{c}(\mathbb{R}\backslash \{0\}) \rightarrow \mathcal{D}'(\mathcal{M} \times \mathcal{M}, \mathcal{E}\boxtimes\mathcal{E}^*) ~.
	\end{equation*}
	Applying the Schwartz kernel theorem once more we therefore obtain as its kernel a distribution $\mathbf{K}_f(x, y, t)\in \mathcal{D}'(\mathcal{M} \times \mathcal{M} \times \mathbb{R}\backslash \{0\}, \mathcal{E}\boxtimes\mathcal{E}^*)$:
	\begin{equation*}
	\mathbf{A}_{f, \chi}\big( \mathbf{u} \big)(x) = \int_{\mathcal{M}\times \mathbb{R}} \mathbf{K}_f(x, y, t) \mathbf{u}(y) \chi(t) \mathrm{d}y \mathrm{d}t ~,
	\end{equation*}
	where $\mathrm{d}y$ is the same (fixed but arbitrary) density on $\mathcal{M}$ used to define the kernel of $\mathbf{A}_{f, \chi}$ and $\mathbf{u}\in \mathrm{C}^\infty(\mathcal{M}, \mathcal{E})$ is any smooth section of the bundle $\mathcal{E}$.
	
	At this point the classical \cite[Thm.~8.2.12]{Hoerm.2008} immediately shows that
	\begin{equation} \label{eq_open_systems5}
	\mathrm{WF}(\mathbf{K}_{f, \chi}) \subseteq \set{(x, y, \xi, \eta)}{\exists t\in \mathrm{supp}(\chi) \,\text{with}\, (x, y, t, \xi, \eta, 0)\in \mathrm{WF}(\mathbf{K}_f)} ~,
	\end{equation}
	where $\mathbf{K}_{f, \chi}\in \mathcal{D}'(\mathcal{M}\times \mathcal{M}, \mathcal{E}\boxtimes\mathcal{E}^*)$ denotes the kernel of $\mathbf{A}_{f, \chi}$. As $\mathbf{K}_{f, \chi}$ is compactly supported by virtue of $\widetilde{\chi}$, we are left with the task of estimating the wavefront set of the kernel $\mathbf{K}_f$.
	
	To do so, first of all note that Equation \eqref{eq_leibniz} implies that $\mathrm{e}^{t\mathbf{X}} \left( f\mathbf{u} \right) = \left(f\circ \varphi_t\right) \left(\mathrm{e}^{t\mathbf{X}} \mathbf{u}\right)$ holds by deriving both sides with respect to $t$ and using uniqueness. Now suppose that $\chi\otimes \mathbf{u}\otimes \mathbf{v}\in \mathrm{C}^\infty_\mathrm{c}(I) \otimes \mathrm{C}^\infty(\mathcal{M}, \mathcal{E})\otimes \mathrm{C}^\infty(\mathcal{M}, \mathcal{E})$ is supported in a small coordinate patch $U\times U'\subseteq\mathcal{M}\times \mathcal{M}$ and an open $I\subseteq \mathbb{R}\backslash\{0\}$. Then we may assume $\mathbf{u} = u^i \mathbf{e}_i,\, \mathbf{v} = v^i \mathbf{e}'_i$ for local (orthonormal) frames $\{\mathbf{e}_i\}$ on $U$ and $\{\mathbf{e}'_j\}$ on $U'$ which allows us to calculate\footnote{We employ the Bochner integral for families of linear operators $T_t$ which is defined via $\left(\int T_t\mathrm{d}t\right)v \defgr \int (T_t v)\mathrm{d}t$ and such that pointwise evaluation is linear and continuous.}
	\begin{equation*}
	\begin{split}
	\qquad&\quad\langle \mathbf{K}_f, \chi\otimes \mathbf{v}\otimes \mathbf{u}\rangle\\
	&= \int_\mathcal{M} \left\langle\left(\mathbf{A}_{f, \chi}\mathbf{u} \right)(x), \mathbf{v}(x)\right\rangle_\mathcal{E} \mathrm{d}x \\
	&= \int_\mathcal{M} \left\langle \int_\mathbb{R} \chi(t) \widetilde{\chi}(x) \mathrm{e}^{-t\mathbf{X}} \left(f \widetilde{\chi} \mathbf{u}\right)(x) \mathrm{d}t, \mathbf{v}(x)\right\rangle_\mathcal{E} \mathrm{d}x \\
	&= \int_{\mathcal{M}\times \mathbb{R}} \chi(t) \widetilde{\chi}(x) \left\langle\mathrm{e}^{-t\mathbf{X}} \left(f \widetilde{\chi} \mathbf{u}\right)(x), \mathbf{v}(x)\right\rangle_\mathcal{E} \mathrm{d}t \mathrm{d}x \\
	&= \sum_{ij} \int_{\mathcal{M}\times \mathbb{R}} \chi(t) \left(f \widetilde{\chi} u^i \right)\left( \varphi_{-t}(x) \right) \left(\widetilde{\chi}v^j \right)(x) \left\langle\mathrm{e}^{-t\mathbf{X}} \left( \mathbf{e}_i \right), \mathbf{e}'_j\right\rangle_\mathcal{E}(x) ~\mathrm{d}t \mathrm{d}x .
	\end{split}
	\end{equation*}
	
	From this calculation we can already deduce that the Schwartz kernel $\mathbf{K}_f$ is supported on the graph $\Gamma_\varphi\defgr \{(x, \varphi_{-t}(x), t) \,|\, x\in\mathcal{M},\, t\in\mathbb{R} \}$ of $\varphi_{-t}$. To explicitly identify it as a smooth function multiplied with the delta function on $\Gamma_\varphi$ we proceed as follows:\footnote{This is very similar to directly computing that the Schwartz kernel of $\varphi_{-t}^*$ as a distribution in $\mathbb{R}\backslash\{0\}\times \mathcal{M}\times \mathcal{M}$ is a delta distribution on $\{y = \varphi_{-t}(x)\}$: In the scalar case, $\varphi_{-t}^*: \chi(t)\otimes f\otimes g\mapsto \int_{\mathcal{M}\times \mathbb{R}} \chi(t) (f\circ \varphi_{-t})(x) g(x) \mathrm{d}x \mathrm{d}t$ yields the kernel $\psi(t, x, y) \mapsto \int_{\mathcal{M}\times \mathbb{R}} \psi(t, x, \varphi_{-t}(x)) \mathrm{d}x \mathrm{d}t$ for $\varphi_{-t}^*$.} Let $\widetilde{U}\subseteq U$ be a smaller coordinate neighbourhood and fix a cutoff function $\rho\in \mathrm{C}^\infty_\mathrm{c}(U)$ with $\rho\restr{\widetilde{U}} \equiv 1$. Then we define the local component functions $a_{ij}\in \mathrm{C}^\infty(\widetilde{U}\times I)$ of the transfer operator via
	\begin{equation} \label{eq_Kf_coeff}
	\left(\mathrm{e}^{-t\mathbf{X}} \rho \mathbf{e}_j\right)(x) = \sum_{i} a_{ij}(x, t) \mathbf{e}_i(x)~, \qquad x\in U,\, t\in I ~.
	\end{equation}
	We suppress the dependency of the components $a_{ij}$ for the sake of simplicity. After choosing two additional cutoff functions $\phi\in \mathrm{C}^\infty_\mathrm{c}(\widetilde{U})$ and $\psi\in \mathrm{C}^\infty_\mathrm{c}(U')$ (note that $\phi = \rho\phi$), this allows us to calculate the local components of the Schwartz kernel $\mathbf{K}_f$ w.r.t. the local frames $\{\mathbf{e}_i\}$ and $\{\mathbf{e}'_j\}$:
	\begin{equation*}
	\begin{split}
	\qquad&\quad\left(\mathbf{K}_f\right)_{ij}\\
	&\defgr \langle \mathbf{K}_f, \chi\otimes \left(\psi \mathbf{e}'_j\right)\otimes \left(\phi \mathbf{e}_i\right) \rangle \\
	&= \sum_k \int_{U'}\int_I \chi(t) \left(f \widetilde{\chi} \phi \right)\left( \varphi_{-t}(x) \right) \left(\widetilde{\chi}\psi \right)(x) a_{ki}(x, t) \left\langle \mathbf{e}_k, \mathbf{e}'_j \right\rangle_\mathcal{E} (x) ~\mathrm{d}t \mathrm{d}x ~.
	\end{split}
	\end{equation*}
	Now the wavefront set of $\mathbf{K}_f$ is defined locally as the union of the wavefront sets of its component functions $\left(\mathbf{K}_f\right)_{ij}$. But our calculation above shows that these components are delta distributions on the graph of $\varphi_{-t}$, i.e. the submanifold $\Gamma_\varphi$, multiplied by some smooth functions. The final wavefront set is therefore contained in the conormal bundle $N^* \Gamma_\varphi$ given by (\cite[Example~8.2.5]{Hoerm.2008}):
	\begin{equation*}
	N^* \Gamma_\varphi \defgr \set{(x, y, t, \xi, \eta, \tau)\in T^*(\mathcal{M}\times \mathcal{M}\times \mathbb{R})}{y = \varphi_{-t}(x),\, (\xi, \eta, \tau)\restr{T(\Gamma_\varphi)} = 0} ~.
	\end{equation*}
	
	Now the tangent space to a graph is the image of the differential of the defining function. Thus,
	\begin{equation*}
	T_{(x, \varphi_{-t}(x), t)} (\Gamma_\varphi) = \set{(v, (\mathrm{d} \varphi_{-t}(x))v - \theta X_{\varphi_{-t}(x)}, \theta)}{v\in T_x\mathcal{M},\, \theta\in T_t \mathbb{R} = \mathbb{R}} ~,
	\end{equation*}
	and we immediately get that $\mathrm{WF}(\mathbf{K}_f)$ is contained in
	\begin{equation*}
	\begin{split}
	\big\{(x, y, t, \xi, \eta, \tau)\in T^* ( \mathcal{M} \times \mathcal{M} \times \mathbb{R} ) \,\big|\, y &= \varphi_{-t}(x),\, \eta\neq 0,\, \xi = -(\mathrm{d}\varphi_{-t}(x))^{T}\eta, \\
	\tau &= \langle \eta, X_y\rangle,\, x,y\in \mathrm{supp}(\widetilde{\chi}) \big\} ~.
	\end{split}
	\end{equation*}
	
	Now we are ready to substitute into \eqref{eq_open_systems5} to estimate the wavefront set of the right-hand side of the trace formula as follows:
	\begin{equation} \label{eq_open_systems6}
	\begin{split}
	\mathrm{WF}(\mathbf{K}_{f, \chi}) \subseteq \big\{(x, y, \xi, \eta) &\in T^* ( \mathcal{M} \times \mathcal{M} ) \,\big|\, \exists t\in\mathrm{supp}(\chi)\, x,y\in \mathrm{supp}(\widetilde{\chi}):\\
	y &= \varphi_{-t}(x),\, \eta\neq 0,\, \xi = -(\mathrm{d}\varphi_{-t}(x))^{T}\eta,\, \langle \eta, X_y\rangle = 0 \big\} ~.
	\end{split}
	\end{equation}
	
	This set does not intersect the set $\set{(x, x, \xi, -\xi)}{\xi\in T^*_x\mathcal{M}}$: Any $(x, x, \xi, -\xi)$ contained in the right-hand side of \eqref{eq_open_systems6} would satisfy $\varphi_T(x) = x$, for some $T\neq 0$, and $\langle \xi, X_x\rangle = 0$ together with $(\mathrm{id} - \mathrm{d}\varphi_{-T}(x)^{T}) \xi = 0$. But $\mathrm{id} - \mathrm{d}\varphi_{-T}(x)$ is invertible on $E_u(x)\oplus E_s(x)$ and $T_x\mathcal{M} = E_u(x)\oplus E_s(x)\oplus X_x$ for any $x$ on a closed geodesic. We can therefore conclude $\xi = 0$, which does not belong to $\mathrm{WF}(\mathbf{K}_{f, \chi})$. Thus, the flat trace on the left-hand side is well defined.

	\subsection{Local Trace Formula} \label{trace.2}
	
	Denoting by $i: \mathcal{M}\times (\mathbb{R}\backslash \{0\}) \rightarrow \mathcal{M}\times\mathcal{M} \times (\mathbb{R}\backslash \{0\})$ the inclusion $(x, t) \mapsto (x, x, t)$ we observe that the (by Section \ref{trace.1} well-defined) distribution $i^* \mathbf{K}_f$ is supported within
	\begin{equation*}
	i^{-1}(\mathrm{supp}(\mathbf{K}_f)) \subseteq \set{(x_0, T)}{\varphi_T(x_0) = x_0,\, T\neq 0} ~.
	\end{equation*}
	
	As in \cite[Lemma~B.1]{Dyatlov.2016}, it therefore makes sense to prove the following local lemma which will be the key ingredient for the final trace formula:
	\begin{lemma}{Local Weighted Trace Formula}{local_trace}
		Suppose $x_0\in \mathcal{U}$ and $T\neq 0$ is such that $\varphi_T(x_0) = x_0$. Then there are $\varepsilon > 0$ and an open neighbourhood $x_0\in U\subseteq \mathcal{U}$ with $\varphi_s(x_0)\in U$ for any $\vert s\vert < \varepsilon$ and such that for any $\rho(x, t) = \sigma(x)\chi(t) \in \mathrm{C}^\infty_\mathrm{c}(U\times ]T - \varepsilon, T + \varepsilon[)$ the following holds:
		\begin{equation}
		\begin{split}
		\mathrm{tr}^\flat\left( \int_\mathbb{R}\rho(x, t) \mathbf{K}_f(x, y, t) \mathrm{d}t \right) &= \int_{\mathbb{R}\times \mathcal{M}} \rho(x, t) \mathbf{K}_f(x, x, t) \mathrm{d}t \mathrm{d}x \\
		&= \frac{\mathrm{tr}(\alpha_\gamma)}{\vert \det(\mathrm{id} - \mathcal{P}_\gamma)\vert} \int_{-\varepsilon}^\varepsilon \rho(\varphi_s(x_0), T) f(\varphi_s(x_0)) \mathrm{d}s
		\end{split}
		\end{equation}
	\end{lemma}
	
	\begin{proof}
		First of all, according to \cite[Thm.~9.22]{Lee.2012} there exists a chart $\kappa: U_1\subseteq \mathcal{M} \rightarrow \mathrm{B}^n_{\varepsilon_1}(0),\, \kappa(x) = w$ around $x_0$ such that
		\begin{equation*}
		\kappa(x_0) = 0,\qquad \kappa_*(X) = \partial_{w_1} ~.
		\end{equation*}
		By composing $\kappa$ with a suitable linear map we can also assume that
		\begin{equation*}
		\mathrm{d}\kappa(x_0)\left(E_s(x_0) \oplus E_u(x_0) \right) = \{\mathrm{d}w_1 = 0\} ~.
		\end{equation*}
		By the joint continuity of the flow $\varphi$ in the variables $(x, t)$ we can find some $\varepsilon > 0$ such that $U\defgr \kappa^{-1}(\mathrm{B}^n_\varepsilon(0))$ satisfies $\varphi_{-t}(U) \subseteq U_1$ for any $\vert T - t\vert < \varepsilon$.
		
		Next, define two functions $A: \mathrm{B}^{n - 1}_\varepsilon(0) \rightarrow \mathrm{B}^{n - 1}_{\varepsilon_1}(0)$ and $F: \mathrm{B}^{n - 1}_\varepsilon(0) \rightarrow ]-\varepsilon_1, \varepsilon_1[$ by the relation
		\begin{equation*}
		\kappa\circ \varphi_{-T}\left( \kappa^{-1}(0, w') \right) = \left( F(w'), A(w') \right), \qquad \text{for}~ w'\in \mathbb{R}^{n - 1},\, \vert w'\vert < \varepsilon ~.
		\end{equation*}
		Obviously, $F(0) = 0$ and $A(0) = 0$. By the assumption $\kappa_*(X) = \partial_{w_1}$ we get that $\varphi_{-T}\left( \kappa^{-1}(w_1, w') \right) = \kappa^{-1}\left( w_1 + F(w'), A(w') \right)$ for any $(w_1, w')\in \mathrm{B}^{n}_\varepsilon(0)$. Similarly we get for $\vert T - t\vert < \varepsilon$ that
		\begin{equation*}
		\varphi_{-t}\left( \kappa^{-1}\left(w_1, w'\right) \right) = \kappa^{-1}\left( T - t + w_1 + F(w'), A(w') \right) ~.
		\end{equation*}
		
		With this local setup we can calculate the flat trace explicitly. First, we assume that the density $\mathrm{d}x$ used to define $\mathbf{K}_f$ is translated into the standard density on $\mathbb{R}^n$ under $\kappa$.\footnote{One may do this, as the flat trace is independent of the choice of density (c.f. \cite[Section~2.4]{Dyatlov.2016}).} The tricky part of calculating the pullback of $\sigma \mathbf{K}_{f, \chi}$ along the inclusion $\iota: \mathcal{M}\ni x\mapsto (x, x)\in \mathcal{M}\times \mathcal{M}$ is given by the singular portion of $\mathbf{K}_{f, \chi}$. We therefore start by calculating the pullback of $\delta_{y = \varphi_{-t}(x)}$ along $\widetilde{\kappa}^{-1}: (w, t)\mapsto (\kappa^{-1}(w), \kappa^{-1}(w), t)$.
		
		To this end let $\psi \in \mathrm{C}^\infty_\mathrm{c}(\mathrm{B}^n_\varepsilon(0) \times \mathrm{B}^n_\varepsilon(0)\times ]-\varepsilon, \varepsilon[)$ be an arbitrary test function and observe:
		\begin{equation} \label{eq_local_delta}
		\begin{split}
		&\left\langle \left(\widetilde{\kappa}^{-1}\right)^* \rho\delta_{y = \varphi_{-t}(x)}, \psi\right\rangle \\
		= &\int_{\mathrm{B}^n_\varepsilon(0)} \int_\mathbb{R} \rho(\kappa^{-1}(w), t) \psi(w, \kappa\circ \varphi_{-t}\circ \kappa^{-1}(w), t) \mathrm{d}w \mathrm{d}t \\
		= &\int_{\mathrm{B}^n_\varepsilon(0)} \int_\mathbb{R} \rho(\kappa^{-1}(w_1, w'), t) \psi((w_1, w'), (T - t + w_1+ F(w'), A(w')), t) \mathrm{d}w_1 \mathrm{d}w' \mathrm{d}t ~.
		\end{split}
		\end{equation}
		In summary, Equation \ref{eq_local_delta} lets us conclude that in coordinates $\rho\delta_{y = \varphi_{-t}(x)}$ equals $\rho(\kappa^{-1}(w), t) \delta(z_1 - T + t - w_1 - F(w')) \delta(z' - A(w'))$, with $z = (z_1, z')$. Distributions of this form are well known. In particular we have $\iota^* \delta(y - f(x)) = \delta(x - f(x)) = \vert \mathrm{det}(\mathrm{id} - f'(0)) \vert^{-1} \delta_0$, if $f(x) = x$ is equivalent to $x = 0$ and $f'(0)$ has no eigenvalue equal to one.
		
		To apply this knowledge we have to look at the solutions of $A(w') = w'$. To this end we calculate
		\begin{equation*}
		\varphi_{-(T + F(w'))}\left( \kappa^{-1}(0, w') \right) = \kappa^{-1}\left( T - T - F(w') + F(w'), A(w') \right) = \kappa^{-1}(0, A(w')) ~,
		\end{equation*}
		which shows that $A(w') = w'$ entails that $\kappa^{-1}(0, w')$ lies on a closed trajectory with period $T + F(w')$. But we can choose $\varepsilon$ small enough that $U$ does not intersect any closed trajectory with period in $[T - \varepsilon, T + \varepsilon]$ apart from $\varphi_t(x_0)$. Then $A(w') = w'$ implies $(0, w') = \kappa(x_0) = (0, 0)$ and indeed $w' = 0$. Combining these results and writing $\delta(x, y) = \int \chi(t) \delta_{y = \varphi_{-t}(x)} \mathrm{d}t$ yields
		\begin{equation} \label{eq_sing_part}
		\begin{split}
		\left\langle \iota^* \left(\sigma(x)\delta(x, y)\right), \mathbf{1} \right\rangle &= \int_{\mathrm{B}^n_\varepsilon(0)} \rho(\kappa^{-1}(w_1, w'), T + F(w')) \delta(w' - A(w')) \mathrm{d}w_1 \mathrm{d}w' \\
		&= \frac{1}{\vert \mathrm{det}(\mathrm{id} - A'(0)) \vert}\int_{-\varepsilon}^\varepsilon \rho(\kappa^{-1}(w_1, 0), T) \mathrm{d}w_1 \\
		&= \frac{1}{\vert \mathrm{det}(\mathrm{id} -\mathcal{P}_\gamma) \vert}\int_{-\varepsilon}^\varepsilon \rho(\varphi_s(x_0), T) \mathrm{d}s ~,
		\end{split}
		\end{equation}
		where the last equality comes from the fact that $A'(0)$ is conjugated to $\mathcal{P}_\gamma$, where $\gamma$ is the closed geodesic with period $T$ containing $x_0$, via the map $\mathrm{d} \kappa(x_0)$ and $\varphi_{s - T}(x_0) = \kappa^{-1}(s, 0)$. Note that we can integrate over $[0, T_\gamma^\#]$ instead of $[-\varepsilon, \varepsilon]$ as $\rho = 0$ outside the arc $\set{\varphi_s(x_0)}{s\in [-\varepsilon, \varepsilon]} = \set{\varphi_s(x_0)}{s\in [0, \varepsilon]\cup [T_\gamma^\# - \varepsilon, T_\gamma^\#]}$.
		
		In the final step we combine this result with the additional smooth factors appearing in the actual kernel $\mathbf{K}_{f, \chi}$: First, we have additional cut-off functions $\widetilde{\chi}$ and the weight function $f$, but these can be taken care of by substituting $\widetilde{\chi} \rho \widetilde{\chi} f$ instead of $\rho$ everywhere. The final integral will then contain $\rho f$ instead of $\rho$ as $\widetilde{\chi} = 1$ on the trapped set $K$.
		
		Finally, the flat trace of the $\mathcal{E}\boxtimes \mathcal{E}^*$-valued kernel $\mathbf{K}_f$ contains the trace over the local matrix coefficients defined in \eqref{eq_Kf_coeff}, i.e. the smooth function $\sum_i a_{ii}(\varphi_s(x_0), T)$ inside the integral. But given a local frame $\{\mathbf{e}_i\}$ we immediately calculate that
		\begin{equation*}
		\alpha_{x, t}: \mathbf{e}_i(x) \longmapsto \left(\mathrm{e}^{-t \mathbf{X}} \mathbf{e}_i \right) (\varphi_t(x)) = \sum_j a_{ji}(\varphi_t(x), t) \mathbf{e}_j(\varphi_t(x)) ~,
		\end{equation*}
		i.e. the matrix representing $\alpha_{x_0, T}$ in the basis $\{\mathbf{e}_i(x_0)\}$ is $\left(a_{ij}(\varphi_T(x_0), T) \right)$. We therefore have
		\begin{equation*}
		\sum_i a_{ii}(x_0, T) = \mathrm{tr}(\alpha_\gamma) = \sum_i a_{ii}(\varphi_s(x_0), T)
		\end{equation*}
		for all $s$, because the $\alpha_{\gamma(s), T}$ are conjugate to each other. Plugging this last ingredient into \eqref{eq_sing_part} yields the local trace formula.
	\end{proof}

	\subsection{Completing the Proof} \label{trace.3}

	Now a \textit{partition of unity} argument finishes the proof of \Cref{lem:guillemin_trace}: Let $\chi\in \mathrm{C}^\infty_\mathrm{c}(\mathbb{R}\backslash\{0\})$ and denote by $\mathcal{L}(\chi)$ the (finite) set of all closed trajectories $\gamma$ of $\varphi_t$ with periods $T_\gamma$ contained in $\mathrm{supp}(\chi)$.
	
	By compactness of $\gamma$ we can now cover any $\gamma \times \{T_\gamma\}$, $\gamma\in \mathcal{L}(\chi)$, with finitely many $U_{i, \gamma}\times ]T_\gamma - \varepsilon_{i, \gamma}, T_\gamma + \varepsilon_{i, \gamma}[$ according to \Cref{lem:local_trace}. Given a partition of unity $\chi_{i, \gamma}(x)$ subordinate to $\{U_{i, \gamma}\}$ we calculate
	\begin{equation*}
	\begin{split}
	\mathrm{tr}^\flat\left( \mathbf{K}_{f, \chi} \right) &= \sum_{\gamma\in \mathcal{L}(\chi)} \sum_i \int_{\mathbb{R}\times \mathcal{M}} \chi(t) \chi_{i, \gamma}(x) \mathbf{K}_f(x, x, t) \mathrm{d}t \mathrm{d}x \\
	&= \sum_{\gamma} \frac{\mathrm{tr}(\alpha_\gamma) \chi(T_\gamma)}{\vert \det(\mathrm{id} - \mathcal{P}_\gamma)\vert} \sum_i \int_{-\varepsilon_{i, \gamma}}^{\varepsilon_{i, \gamma}} \chi_{i, \gamma}(\varphi_s(x_{i, \gamma})) f(\varphi_s(x_{i, \gamma})) \mathrm{d}s
	\end{split}
	\end{equation*}
	Finally, looking only at a fixed geodesic $\gamma$, we can drop the second subscript and write, by using the semi-group property of the flow,
	\begin{equation*}
	\begin{split}
	\sum_i \int_{-\varepsilon_i}^{\varepsilon_i} \chi_i(\varphi_s(x_i)) f(\varphi_s(x_i)) \mathrm{d}s &= \sum_i \int_0^{T_\gamma^\#} \chi_i(\varphi_s(x_0)) f(\varphi_s(x_0)) \mathrm{d}s \\
	&= \int_0^{T_\gamma^\#} f(\varphi_s(x_0)) \mathrm{d}s = \int_{\gamma^\#} f ~,
	\end{split}
	\end{equation*}
	which follows from the fact that the $U_i$ cover $\gamma^\#$, $\varphi_\cdot(x_0)$ is a diffeomorphism from the torus of circumference $T_\gamma^\#$ onto $\gamma^\#$, and the $\chi_i$ are compactly supported on an arc of $\gamma$.
\end{proof}
	

\section{Proof of the Main Theorem} \label{proof}

In this section we finally prove our main theorem, namely the \emph{meromorphic continuation} of the following weighted zeta function defined in terms of closed trajectories $\gamma$ of an open hyperbolic system $(\mathcal{U}, \varphi_t, \mathbf{X})$ and a weight $f\in\mathrm{C}^\infty(\mathcal{U})$:
\begin{equation*}
Z_f^\mathbf{X}(\lambda) \defgr \sum_{\gamma} \left( \frac{\exp\left(-\lambda T_\gamma\right) \mathrm{tr}(\alpha_\gamma)}{\vert \det(\mathrm{id} - \mathcal{P}_\gamma) \vert} \int_{\gamma^\#} f \right), \qquad \lambda\in \C.
\end{equation*}
In Section \ref{proof.1} we state and prove this meromorphic continuation result using the trace formula of the previous Section \ref{trace}. Now in practice the defining properties of open hyperbolic systems are rather cumbersome to handle. We therefore remove these requirements in Section \ref{proof.2} and obtain the result cited in the introduction.

\subsection{Proof for Open Hyperbolic Systems} \label{proof.1}

We prove meromorphic continuation together with an explicit formula for the Laurent coefficients of $Z_f$:

\begin{theorem}{Meromorphic Continuation of Weighted Zetas I}{analyt_zeta1}		
	The weighted zeta function $Z_f^\mathbf{X}$ defined in \ref{eq_zeta_def} for open hyperbolic systems converges absolutely in $\{\mathrm{Re}(\lambda) \gg 1\}$ and continues meromorphically to $\{\lambda\in \mathbb{C}\}$. Any pole $\lambda_0$ of $Z_f$ is a Pollicott-Ruelle resonance of $\mathbf{X}$ and if $\lambda_0$ has order $J(\lambda_0)$ then for $k\leq J(\lambda_0)$ we have
	\begin{equation*}
	\mathrm{Res}_{\lambda = \lambda_0} \left[ Z_f^\mathbf{X}(\lambda)(\lambda - \lambda_0)^k \right] = \mathrm{tr}^\flat \left((\mathbf{X} - \lambda_0)^k \Pi_{\lambda_0} f \right) .
	\end{equation*}
\end{theorem}

\begin{proof}
	First, we prove that the formal expression \eqref{eq_zeta_def} defines a holomorphic function in $\{\mathrm{Re}(\lambda) \gg 0\}$ by showing uniform convergence on compact sets. To this end, we treat every term separately and then combine the results for a final estimate:
	\begin{enumerate}
		\item[(1.)] First of all, note that $\mathrm{N}(T) \defgr \vert \{ \gamma \,|\, T_\gamma \leq T\} \vert \leq C_0\mathrm{e}^{C_1 T}$ for constants $C_0, C_1 > 0$ according to \cite[Lemma~1.17]{Dyatlov.2016a}.
		\item[(2.)] $\vert \mathrm{det}(\mathrm{id} - \mathcal{P}_\gamma)\vert$ is bounded below by a positive constant $C_2$ as the converse would contradict the existence of a uniform contraction/expansion constant $\gamma > 0$.
		\item[(3.)] $\vert \mathrm{tr}(\alpha_\gamma) \vert \leq C_4 \mathrm{e}^{C_3 T_\gamma}$ by the operator norm estimate on $\mathrm{e}^{-t \mathbf{X}}$.
	\end{enumerate}
	Combining (1.), (2.) and (3.) we get
	\begin{equation*}
	\begin{split}
	\sum_\gamma \left\vert\frac{\mathrm{e}^{-\lambda T_\gamma} \mathrm{tr}(\alpha_\gamma)}{\vert \det(\mathrm{id} - \mathcal{P}_\gamma) \vert} \int_{\gamma^\#} f \right\vert &\leq \sum_{n\in\mathbb{N}} \sum_{T_\gamma\in ]n-1, n]} \left\vert\frac{\mathrm{e}^{-\lambda T_\gamma} \mathrm{tr}(\alpha_\gamma)}{\vert \det(\mathrm{id} - \mathcal{P}_\gamma) \vert} \int_{\gamma^\#} f \right\vert \\
	&\leq \sum_{n\in\mathbb{N}} C_0 \mathrm{e}^{C_1 n} \vert\mathrm{e}^{-(n - 1) \lambda}\vert C_4 \mathrm{e}^{C_3 n} C_2^{-1} n \arrowvert f\arrowvert_K \\
	&\leq C \sum_{n\in\mathbb{N}} n\cdot \left(\mathrm{e}^{C - \mathrm{Re(\lambda)}}\right)^n
	\end{split}
	\end{equation*}
	$Z_f(\lambda)$ thus converges uniformly if $\lambda$ varies in a compact subset of $\mathrm{Re}(\lambda) > C$. In conclusion, the function $Z_f(\lambda)$ is holomorphic on some right halfplane.
	
	The proof proceeds by expressing the weighted zeta function as the flat trace of an expression involving the \textit{restricted resolvent} and using the trace formula presented in Section \ref{trace} as the main tool. The presentation closely follows \cite[§4]{Dyatlov.2016}. We begin by choosing $0 < t_0 < T_\gamma\, \forall \gamma$, $\chi_T\in \mathrm{C}^\infty_\mathrm{c}(]t_0/2, T + 1[)$ and $\chi_T\equiv 1$ on $[t_0, T]$. Furthermore, we assume $t_0$ small enough such that $\varphi_{-t_0}(\mathrm{supp}(\widetilde{\chi})) \subseteq \mathcal{U}$. Then we can define the family of operators
	\begin{equation*}
	B_T \defgr \int_0^\infty \chi_T(t) \mathrm{e}^{-\lambda t} \left( \widetilde{\chi} \mathrm{e}^{-t\mathbf{X}} \widetilde{\chi} f \right) \mathrm{d}t ~,
	\end{equation*}
	and \Cref{lem:guillemin_trace} shows that, for $\mathrm{Re}(\lambda) \gg 1$,
	\begin{equation*}
	\begin{split}
	\lim\limits_{T\rightarrow \infty} \mathrm{tr}^\flat \left( B_T \right) &= \lim\limits_{T\rightarrow \infty} \sum_\gamma \frac{\chi_T(T_\gamma) \mathrm{e}^{-\lambda T_\gamma} \mathrm{tr}(\alpha_\gamma)}{\vert \mathrm{det}(\mathrm{id} - \mathcal{P}_\gamma) \vert} \int_{\gamma^\#} f \\
	&= \sum_\gamma \frac{\mathrm{e}^{-\lambda T_\gamma} \mathrm{tr}(\alpha_\gamma)}{\vert \mathrm{det}(\mathrm{id} - \mathcal{P}_\gamma) \vert} \int_{\gamma^\#} f ~,
	\end{split}
	\end{equation*}
	because the right-hand side series converges uniformly for $\mathrm{Re}(\lambda) \gg 0$ by virtue of the exponential growth of the number of closed trajectories $\mathrm{N}(T)$.
	
	Now by \cite[Lemma~2.8]{Dyatlov.2016} there exists a family of smoothing operators\footnote{We can directly apply the scalar-valued version of \cite{Dyatlov.2016} by choosing our $E_\varepsilon$ to be diagonal in the fiber variable.} $E_\varepsilon \in \Psi^{-\infty}(\mathcal{M}, \mathrm{End}(\mathcal{E}))$, $\varepsilon > 0$, such that
	\begin{equation*}
	\mathrm{tr}^\flat\left( B_T \right) = \lim\limits_{\varepsilon\rightarrow 0} \mathrm{tr}\left( E_\varepsilon B_T E_\varepsilon \right) ~,
	\end{equation*}
	and the kernels $E_\varepsilon(x, y)$ are supported in $\{(x, y)\in\mathcal{M}\times \mathcal{M} \,|\, \mathrm{d}(x, y) < \varepsilon\}$ for some fixed $c_1 > 0$ and some smooth distance function $\mathrm{d}(\cdot, \cdot)$. Here, the right-hand side trace is taken in $\mathrm{L}^2(\mathcal{M}, \mathcal{E})$. It exists as $E_\varepsilon B_T E_\varepsilon$ is smoothing and therefore trace class by the compactness of $\mathcal{M}$. Now consider the following splitting of the integral over the $t$-variable (note, that the $\mathrm{L}^2$-trace is a bounded linear operator and therefore commutes with $\int \mathrm{d}t$):
	\begin{equation*}
	\begin{split}
	\mathrm{tr}(E_\varepsilon B_T E_\varepsilon) &= \int_{t_0/2}^{t_0} \chi_T(t) \mathrm{e}^{-\lambda t} \mathrm{tr}\left( E_\varepsilon \widetilde{\chi} \mathrm{e}^{-t\mathbf{X}} \widetilde{\chi} f E_\varepsilon \right) \mathrm{d}t \\
	&\quad + \int_{t_0}^\infty \chi_T(t) \mathrm{e}^{-\lambda t} \mathrm{tr}\left( E_\varepsilon \widetilde{\chi} \mathrm{e}^{-t\mathbf{X}} \widetilde{\chi} f E_\varepsilon \right) \mathrm{d}t ~.
	\end{split}
	\end{equation*}
	As in \cite[§4]{Dyatlov.2016}, the first summand vanishes for $\varepsilon$ small enough, as
	\begin{equation*}
	\begin{split}
	&\qquad\qquad \mathrm{tr}\left( E_\varepsilon \widetilde{\chi} \mathrm{e}^{-t\mathbf{X}} \widetilde{\chi} f E_\varepsilon \right) \\
	&= \int_{\mathcal{M}\times \mathcal{M}} E_\varepsilon(x, y) \widetilde{\chi}(y) \left(\sum_i a_{ii}(y, t) \right) f(\varphi_{-t}(y)) \widetilde{\chi}(\varphi_{-t}(y)) E_\varepsilon(\varphi_{-t}(y), x) \mathrm{d}y \mathrm{d}x ~,
	\end{split}
	\end{equation*}
	and $E_\varepsilon(x, y) = 0$ if $\mathrm{d}(x, y) \geq c_1 \varepsilon$ by the support property of the kernel $E_\varepsilon(x, y)$. By $t_0 < T_\gamma$ for all closed $\gamma$ the minimum $c_2$ of $\mathrm{d}(\varphi_{-t}(x), x)$ for $t\in [t_0/2, t_0]$ and $x\in \mathcal{M}$ is strictly positive; if we choose $\varepsilon < c_2 / (2 c_1)$ then
	\begin{equation*}
	E_\varepsilon(x, y) E_\varepsilon(\varphi_{-t}(y), x) \neq 0
	\end{equation*}
	would imply $\mathrm{d}(x, y) < c_1\varepsilon$ \textit{and} $\mathrm{d}(\varphi_{-t}(y), x) < c_1\varepsilon$, i.e. $\mathrm{d}(y, \varphi_{-t}(y)) < 2 c_1 \varepsilon < c_2$, a contradiction. The integrand and therefore the trace vanishes for $t\in [t_0/2, t_0]$.
	
	Next, we can exchange the limits in $T$ and $\varepsilon$, again for $\mathrm{Re}(\lambda)$ sufficiently large,
	\begin{equation*}
	\lim\limits_{T\rightarrow \infty} \mathrm{tr}^\flat\left( B_T \right) = \lim\limits_{T\rightarrow \infty} \lim\limits_{\varepsilon\rightarrow 0} \mathrm{tr}\left( E_\varepsilon B_T E_\varepsilon \right) = \lim\limits_{\varepsilon\rightarrow 0} \lim\limits_{T\rightarrow \infty} \mathrm{tr}\left( E_\varepsilon B_T E_\varepsilon \right) ~,
	\end{equation*}
	as the limit $\varepsilon \rightarrow 0$ exists for any finite $T$ and the limit $T\rightarrow \infty$ is uniform in $\varepsilon > 0$. To see the latter, simply observe that the same calculation as in \cite[Lemma~4.1]{Dyatlov.2016} shows
	\begin{equation} \label{eq_estimate1}
	\int_T^{T + 1} \left\vert \mathrm{tr}\left( \mathrm{e}^{-\lambda t} E_\varepsilon \widetilde{\chi} \mathrm{e}^{-t \mathbf{X}} \widetilde{\chi} f E_\varepsilon \right) \right\vert \mathrm{d}t \leq C \mathrm{e}^{-\mathrm{Re}(\lambda) T} \mathrm{e}^{C T} ~,
	\end{equation}
	where the constant $C$ especially contains the supremum of $f$ on the compact support of $\widetilde{\chi}$ and the supremum of the local fiber traces $\sum_i a_{ii}(x, t)$, which in the $t$-variable can be estimates by $C \mathrm{e}^{C' T}$ due to \cite[Eq.~(0.9)]{Dyatlov.2016a}. Now \eqref{eq_estimate1} implies
	\begin{equation*}
	\begin{split}
	&\left\vert \int_{t_0}^\infty \mathrm{e}^{-\lambda t} \mathrm{tr}\left( E_\varepsilon \widetilde{\chi} \mathrm{e}^{-t \mathbf{X}} \widetilde{\chi} f E_\varepsilon \right) \mathrm{d}t - \int_{t_0}^\infty \chi_T(t) \mathrm{e}^{-\lambda t} \mathrm{tr}\left( E_\varepsilon \widetilde{\chi} \mathrm{e}^{-t \mathbf{X}} \widetilde{\chi} f E_\varepsilon \right) \mathrm{d}t \right\vert \\
	&\qquad \leq \sum_{n = 0}^\infty \int_{T + n}^{T + n + 1} \mathrm{e}^{-\lambda t} \left\vert \mathrm{tr}\left( E_\varepsilon \widetilde{\chi} \mathrm{e}^{-t \mathbf{X}} \widetilde{\chi} f E_\varepsilon \right) \right\vert \mathrm{d}t \\
	&\qquad \leq C \mathrm{e}^{(-\mathrm{Re}(\lambda) + C) T} \sum_{n = 0}^\infty \mathrm{e}^{(-\mathrm{Re}(\lambda) + C) n} \leq \widetilde{C} \mathrm{e}^{(-\mathrm{Re}(\lambda) + C) T} ~,
	\end{split}
	\end{equation*}
	which indeed converges for $T\rightarrow \infty$ and uniformly in $\varepsilon > 0$ (an immediate adaptation of the first equation in \cite[Lemma~4.1]{Dyatlov.2016} shows that the integral which is the limit as $T\rightarrow \infty$ converges absolutely for every $\varepsilon > 0$).
	
	Using the commutativity of the two limits and the absolute convergence of the integral in the $t$-variable we finally arrive at the expression
	\begin{equation*}
	\lim\limits_{T\rightarrow \infty} \mathrm{tr}^\flat\left( B_T \right) = \lim_{\varepsilon \rightarrow 0} \int_{t_0}^\infty \mathrm{e}^{-\lambda t} \mathrm{tr}\left( E_\varepsilon \widetilde{\chi} \mathrm{e}^{-t \mathbf{X}} \widetilde{\chi} f E_\varepsilon \right) \mathrm{d}t ~,
	\end{equation*}
	where $\mathrm{Re}(\lambda) \gg 1$. We therefore have
	\begin{equation} \label{eq_zeta_resolvent}
	\begin{split}
	Z_f^\mathbf{X}(\lambda) &\defgr \sum_\gamma \frac{\mathrm{e}^{-\lambda T_\gamma} \mathrm{tr}(\alpha_\gamma)}{\vert \mathrm{det}(\mathrm{id} - \mathcal{P}_\gamma) \vert} \int_{\gamma^\#} f = \lim_{T\rightarrow \infty} \mathrm{tr}^\flat(B_T) \\
	&= \lim_{\varepsilon \rightarrow 0} \int_0^\infty \mathrm{e}^{-\lambda (t + t_0)} \mathrm{tr}\left( E_\varepsilon \widetilde{\chi} \mathrm{e}^{-(t + t_0) \mathbf{X}} \widetilde{\chi} f E_\varepsilon \right) \mathrm{d}t \\
	&= \mathrm{e}^{-\lambda t_0} \lim_{\varepsilon \rightarrow 0} \mathrm{tr}\left( E_\varepsilon \widetilde{\chi} \mathrm{e}^{-t_0 \mathbf{X}} \left(\int_{0}^\infty \mathrm{e}^{-\lambda t} \mathrm{e}^{-t \mathbf{X}} \mathrm{d}t \right) \widetilde{\chi} f E_\varepsilon \right) \\
	&= \mathrm{e}^{-\lambda t_0} \lim_{\varepsilon \rightarrow 0} \mathrm{tr}\left( E_\varepsilon \widetilde{\chi} \mathrm{e}^{-t_0 \mathbf{X}} (\mathbf{X} + \lambda)^{-1} \widetilde{\chi} f E_\varepsilon \right) ~,
	\end{split}
	\end{equation}
	which at first only holds for $\mathrm{Re}(\lambda) \gg 1$. Next we want to apply the meromorphic continuation of the \textit{restricted resolvent} $\mathbf{R}(\lambda) \defgr \mathbf{1}_\mathcal{U} (\mathbf{X} + \lambda)^{-1} \mathbf{1}_\mathcal{U}: \Gamma^\infty_\mathrm{c}(\mathcal{U}, \mathcal{E}) \rightarrow \mathcal{D}'(\mathcal{U}, \mathcal{E})$ achieved in \cite{Dyatlov.2016a}. To do so we first observe that $(\mathbf{X} + \lambda)^{-1} \widetilde{\chi} = (\mathbf{X} + \lambda)^{-1} \mathbf{1}_\mathcal{U} \widetilde{\chi}$ by $\mathrm{supp}(\widetilde{\chi}) \subseteq \mathcal{U}$. But we even demanded $t_0$ to be small enough that $\varphi_{-t_0}(\mathrm{supp}(\widetilde{\chi})) \subseteq \mathcal{U}$ holds, which by the support property of $\mathrm{e}^{-t_0 \mathbf{X}}$ lets us rewrite $\widetilde{\chi} \mathrm{e}^{-t_0 \mathbf{X}} (\mathbf{X} + \lambda)^{-1} \mathbf{1}_\mathcal{U} \widetilde{\chi} = \widetilde{\chi} \mathrm{e}^{-t_0 \mathbf{X}} \mathbf{1}_\mathcal{U} (\mathbf{X} + \lambda)^{-1} \mathbf{1}_\mathcal{U} \widetilde{\chi}$.
	
	Now if $\lambda\in \mathbb{C}$ is not a resonance then the general wavefront estimates \cite[Example~8.2.5]{Hoerm.2008} and \cite[Thm.~8.2.14]{Hoerm.2008} together with the estimate of $\mathrm{WF}'(\mathbf{R}(\lambda))$ in \cite[Equation~(3.43)]{Dyatlov.2016a} and the fact that multiplication with smooth functions does not enlarge the wavefront set yield
	\begin{equation} \label{eq_wavefront_total}
	\begin{split}
	&\qquad \mathrm{WF}'\left( \widetilde{\chi} \mathrm{e}^{-t_0 (\lambda + \mathbf{X})} \mathbf{R}(\lambda) \widetilde{\chi} f \right) \\
	\subseteq& \left\{ \left(\mathrm{e}^{t_0 H_p}(x, \xi), x, \xi\right) \,\big|\, (x, \xi)\in T^*\mathcal{U}\backslash 0 \right\} \cup (E^*_+ \times E^*_-) \\
	&\cup \left\{ \left(\mathrm{e}^{(t_0 + t) H_p}(x, \xi), x, \xi\right) \,\big|\, (x, \xi)\in T^*\mathcal{U}\backslash 0,\, t\geq 0,\, p(x, \xi) = 0 \right\} ~,
	\end{split}
	\end{equation}
	where $p(x, \xi) \defgr \xi(X_x)$, $\mathrm{e}^{t_0 H_p}(x, \xi) = (\varphi_t(x), (\mathrm{d} \varphi_t(x))^{-T} \xi)$ and the wavefront set of $\mathrm{e}^{-t_0 \mathbf{X}}$ is contained in the graph of $\mathrm{e}^{t_0 H_p}$. Now the wavefront set \eqref{eq_wavefront_total} \textit{does not intersect} the diagonal: For the first component this immediately follows from $t_0 < T_\gamma$ for all closed orbits $\gamma$, for the second component this follows because the dual hyperbolic splitting is direct, i.e. $E^*_+(x)\cap E^*_-(x) = \{0\}$ for every $x\in K$. Finally, any vector in the intersection of the third component and the diagonal would satisfy $(\mathrm{d} \varphi_t(x))^{-T} \xi = \xi$ for some $t\geq t_0$, as well as $\xi(X_x) = 0$. This is impossible for any $\xi\neq 0$ as $\mathrm{id} - \mathrm{d} \varphi_t(x)$ is invertible on $E_s(x)\oplus E_u(x)$ and $T_x\mathcal{U} = X_x\oplus E_s(x)\oplus E_u(x)$ for $x\in K$.
	
	Now if $\lambda$ \textit{is} a resonance, then the same argument using \cite[Lemma~3.5]{Dyatlov.2016a} (and linearity of $\mathrm{tr}^\flat$) shows that the right-hand side of \eqref{eq_zeta_resolvent} admits a Laurent expansion around $\lambda$ whose finitely many coefficients of negative order are the flat traces of the coefficients of $\mathbf{R}(\lambda)$. We therefore have a meromorphic continuation for $Z_f^\mathbf{X}(\lambda)$ onto $\{\lambda\in \mathbb{C}\}$, its poles are contained in the set of Pollicott-Ruelle resonances of $\mathbf{X}$ and for $\lambda$ not a resonance we have
	\begin{equation} \label{eq_mero_final}
	Z_f^\mathbf{X}(\lambda) = \mathrm{e}^{-\lambda t_0} \mathrm{tr}^\flat\left( \widetilde{\chi} \mathrm{e}^{-t_0 \mathbf{X}} \mathbf{R}(\lambda) \widetilde{\chi} f \right) ~.
	\end{equation}
	
	To complete our proof we will show an explicit formula for the Laurent coefficients at a resonance $\lambda_0$. The starting point is \eqref{eq_mero_final} combined with (0.13) in \cite[Thm.~2]{Dyatlov.2016a}. If we substitute the expansion in the second equation into the first we get
	\begin{equation*}
	Z_f^\mathbf{X}(\lambda) = Z_{f, H}^\mathbf{X}(\lambda) + \sum_{j = 1}^{J(\lambda_0)} \mathrm{tr}^\flat \left( \widetilde{\chi} \frac{\mathrm{e}^{-t_0 (\lambda + \mathbf{X})} (-\mathbf{X} - \lambda)^{j - 1} \Pi_{\lambda_0}}{(\lambda - \lambda_0)^j} \widetilde{\chi}f \right) ~,
	\end{equation*}
	where $Z_{f, H}^\mathbf{X}(\lambda)$ is holomorphic near $\lambda_0$.
	
	For $0 \leq k < J(\lambda_0)$, one can use the Taylor expansion of the exponential around $\lambda_0$, i.e. $\exp(-\lambda t_0) = \sum_{n = 0}^\infty (-t_0)^n \exp(-\lambda_0 t_0) (\lambda - \lambda_0)^n/n!$, to obtain the weighted zeta function's Laurent coefficient of order $k$ at $\lambda_0$:
	\begin{equation} \label{eq1.18}
	\begin{split}
	\underset{\lambda = \lambda_0}{\mathrm{Res}} \left[Z_f^\mathbf{X}(\lambda)(\lambda - \lambda_0)^k \right] &= \underset{\lambda = \lambda_0}{\mathrm{Res}} \left[\sum_{j = k+1}^{J(\lambda_0)} \sum_{n = 0}^{\infty} \mathrm{tr}^\flat \left(\widetilde{\chi} \frac{(-t_0)^n \mathrm{e}^{-t_0 (\lambda_0 + \mathbf{X})} (-\mathbf{X} - \lambda_0)^{j - 1} \Pi_{\lambda_0}}{n! (\lambda - \lambda_0)^{j - k - n}} \widetilde{\chi}f \right)\right] \\
	&= \sum_{n = 0}^{J(\lambda_0) - k - 1} \frac{(-1)^n t_0^n}{n!} \mathrm{tr}^\flat \left(\widetilde{\chi} \mathrm{e}^{-t_0 (\lambda_0 + \mathbf{X})} (-\mathbf{X} - \lambda_0)^{k + n} \Pi_{\lambda_0} \widetilde{\chi}f \right)
	\end{split}
	\end{equation}
	
	The operator $\mathbf{X} + \lambda$ is nilpotent on the image $\mathrm{im}(\Pi_{\lambda_0})$ by Equations (0.12) and (0.15) in \cite{Dyatlov.2016a}. We can therefore simplify the propagator $\mathrm{e}^{-t_0 (\mathbf{X} + \lambda_0)}$ drastically:
	\begin{equation} \label{eq1.19}
	\mathrm{e}^{-t_0 (\mathbf{X} + \lambda_0)} \bigg|_{\mathrm{Im}(\Pi_{\lambda_0})} = \sum_{m = 0}^{J(\lambda_0) - 1} \frac{t_0^m (-\mathbf{X} - \lambda_0)^m}{m!} \bigg|_{\mathrm{Im}(\Pi_{\lambda_0})} ~.
	\end{equation}
	
	Substitution into Equation \eqref{eq1.18} and usage of the abbreviation $N \defgr J(\lambda_0) - k - 1$ yields:
	\begin{equation} \label{eq1.20}
	\begin{split}
	\underset{\lambda = \lambda_0}{\mathrm{Res}} \left[Z_f^\mathbf{X}(\lambda)(\lambda - \lambda_0)^k \right] &= \sum_{n = 0}^{N} \sum_{m = 0}^{N - n} (-1)^n \frac{t_0^{n + m}}{n! m!} \mathrm{tr}^\flat \left( \widetilde{\chi} (-\mathbf{X} - \lambda_0)^{n + m + k} \Pi_{\lambda_0} \widetilde{\chi}f \right) \\
	&= \sum_{s = 0}^{N} \sum_{n = 0}^{s} (-1)^n \frac{t_0^{s}}{n! (s - n)!} \mathrm{tr}^\flat \left( \widetilde{\chi} (-\mathbf{X} - \lambda_0)^{s + k} \Pi_{\lambda_0} \widetilde{\chi}f \right) ~,
	\end{split}
	\end{equation}
	where the second line is obtained by using the variable $s \defgr n + m$ as a reparametrization of the double sum. A close examination of Equation \eqref{eq1.20} reveals that the binomial theorem can be applied to show, for $s > 0$,
	\begin{equation*}
	\sum_{n = 0}^{s} \frac{(-1)^n}{n! (s - n)!} \frac{s!}{s!} = \frac{1}{s!} \sum_{n = 0}^{s} (-1)^n (1)^{s - n} \binom{s}{n} = 0 ~.
	\end{equation*}
	We therefore have the following formula for the $k$-th Laurent coefficient:
	\begin{equation*}
	\underset{\lambda = \lambda_0}{\mathrm{Res}} \left[Z_f^\mathbf{X}(\lambda)(\lambda - \lambda_0)^k \right] = \mathrm{tr}^\flat \left( \widetilde{\chi} (-\mathbf{X} - \lambda_0)^k \Pi_{\lambda_0} \widetilde{\chi}f \right) ~,
	\end{equation*}
	which finishes our proof because the restriction of the kernel of $\Pi_{\lambda_0}$ to the diagonal is supported in $\Gamma_+\cap \Gamma_- = K$ and $\widetilde{\chi} \equiv 1$ on $K$, i.e. we can drop the cutoff functions.
\end{proof}

\begin{remark}
	Note that by \cite[Eq.~(4.8)]{Dyatlov.2016a} we have $\underset{\lambda = \lambda_0}{\mathrm{Res}} \left[Z_f^\mathbf{X}(\lambda) \right] = \mathrm{rank}(\Pi_{\lambda_0})$ if $f \equiv \mathbf{1}$, the constant function on $\mathcal{U}$. All Laurent coefficients of higher (negative) order vanish for this particular choice of test function.
\end{remark}
\begin{remark}
	For a slightly different formulation of the final formula for the Laurent coefficients, one could replace the flat trace by an ordinary trace in an appropriate anisotropic Sobolev space (\cite[Lemma~4.2]{Dyatlov.2016} and the proof of \cite[Thm.~4]{Dyatlov.2016a}).
\end{remark}

\subsection{Removing Strict Convexity for Zetas} \label{proof.2}

The geometric setup of open hyperbolic systems and the requirement of strict convexity in particular are quite cumbersome to state and difficult to verify in practice. Given an arbitrary flow $\varphi_t$ with compact trapped set $K$ which is hyperbolic on $K$ we therefore remove this requirement via the perturbations of $X$ and $\mathbf{X}$ constructed in Section \ref{setup.2}. This will complete the proof of the claims made in the introduction.

\begin{theorem}{Meromorphic Continuation of Weighted Zetas II}{no_strictly_convex}	
	Let the setting of \Cref{thm:resolvent_without_strict_convexity} and a weight $f\in \mathrm{C}^\infty(\mathcal{M})$ be given. Then the weighted zeta function $Z_f^\mathbf{X}$ defined in \eqref{eq_zeta_def} continues meromorphically onto $\mathbb{C}$ and its Laurent coefficients are given by \eqref{eq:laurent}.
\end{theorem}

\begin{proof}
	Recall the discussion in Section \ref{setup.2} where we constructed a vector field $X_0$ and an operator $\mathbf{X}_0$ which satisfy the requirements of open hyperbolic systems and coincide with $X$ and $\mathbf{X}$ on a neighborhood of the trapped set. It is now clear that the weighted zeta function associated with $X_0$ and $\mathbf{X}_0$ coincides with the weighted zeta function for $X$ and $\mathbf{X}$ as both functions depend only on the dynamics near the trapped set. This yields the claim by an immediate application of \Cref{thm:analyt_zeta}.
\end{proof}

\begin{remark}
	\Cref{thm:no_strictly_convex} implies the independence of the set of resonances from the open sets $\mathcal{U}$ and $\mathcal{U}_0$ promised in Section \ref{setup.2}: Choosing as our weight the constant function $f \equiv \mathbf{1}$ we obtain $\mathrm{Res}_{\lambda = \lambda_0} [Z_\mathbf{1}(\lambda)] = \mathrm{tr}^\flat(\Pi_{\lambda_0}) = \mathrm{rank}(\Pi_{\lambda_0})$, where the second equality was proven in \cite[proof of Thm.~4]{Dyatlov.2016a}. Now this implies that the resonances coincide exactly with the poles of the weighted zeta function with constant weight, but the definition of the latter only involves the dynamics on the trapped set.
\end{remark}


\section{Residue Formula for Patterson-Sullivan Distributions} \label{ps}

In this section we relate the residues of $Z_f^\mathbf{X}$, defined in purely classical terms such as closed trajectories, with certain quantum mechanical phase space distributions called \emph{Patterson-Sullivan} distributions. For the general setup suppose $\mathbf{M}$ is a Riemannian manifold, $\varphi_t$ its geodesic flow on the unit tangent bundle $S\mathbf{M}$, $X$ the geodesic vector field, and $\Delta_\mathbf{M}$ its Laplacian.

If $\mathbf{M}$ is a \emph{compact hyperbolic surface}, i.e. a compact surface of constant negative curvature equal to $-1$, the spectrum of $\Delta_\mathbf{M}$ is purely discrete and consists only of eigenvalues: $\sigma(\Delta_\mathbf{M}) = \{\lambda_i \}$, $0 = \lambda_0 < \lambda_1 \leq ...$ and $\Delta_\mathbf{M} \varphi_i = \lambda_i\varphi_i$ for an orthonormal basis of real-valued eigenfunctions $\{\varphi_i\} \subseteq \mathrm{C}^\infty(\mathbf{M})$. In this setting \cite{Zelditch.2007} associated to any $\varphi_i$ a Patterson-Sullivan distribution $\mathrm{PS}_{\varphi_i}\in \mathcal{D}'(S\mathbf{M})$ with the following properties \cite[Eq.~(1.4) and Thm.~1.1]{Zelditch.2007}:
\begin{equation} \label{eq_ps_properties}
\begin{split}
\left( \varphi_t \right)_* \left(\mathrm{PS}_{\varphi_i} \right) &= \mathrm{PS}_{\varphi_i}\\
\forall a\in\mathrm{C}^\infty(S\mathbf{M}): \quad \widehat{\mathrm{PS}}_{\varphi_i}(a) &= \langle \mathrm{Op}(a) \varphi_i, \varphi_i\rangle_{\mathrm{L}^2(\mathbf{M})} + \mathcal{O} \left(1 / \lambda_i \right) .
\end{split}
\end{equation}
Here $\mathrm{Op}$ denotes some quantization procedure on the classical phasespace $S\mathbf{M}$ and $\widehat{\mathrm{PS}}_{\varphi_i}$ denotes the Patterson-Sullivan distribution normalized to $\langle \widehat{\mathrm{PS}}_{\varphi_i}, \mathbf{1}_{S\mathbf{M}} \rangle = 1$. The expression $\mathrm{W}_{\varphi_i} \defgr \langle \mathrm{Op}(a) \varphi_i, \varphi_i\rangle$ is known under the name \emph{Wigner distribution} and the second equation in \eqref{eq_ps_properties} should be interpreted as the equivalence of Patterson-Sullivan and Wigner distributions in the high-frequency limit. \cite[Def.~4.8, Prop.~4.10]{Hilgert.2012} generalizes the construction of Patterson-Sullivan distributions to arbitrary compact locally symmetric spaces. These distributions still retain (generalizations of) the properties \eqref{eq_ps_properties} \cite[Remark~4.11, Thm.~7.4]{Hilgert.2012}.

Our results concerning the residues of weighted zeta functions provide a new view on these phasespace distributions. Concretely, we can use our main \Cref{thm:analyt_zeta} and the results obtained by \cite{Weich.2021} to prove the following residue formula for Patterson-Sullivan distributions in the closed case:
\begin{theorem}{Patterson-Sullivan Distributions as Residues}{ps_residue}
	Let $\mathbf{M} = \Gamma\backslash G/K$ be a compact Riemannian locally symmetric space of rank one, $\Delta_\mathbf{M}$ its Laplacian and $\varphi_t$ the geodesic flow on $S\mathbf{M}$. Let $\rho > 0$ denote the half-sum of the restricted roots of $G$. Then the following holds:
	
	Given $r > 0$ such that $-\rho + \mathrm{i} r$ is a Ruelle resonance of $\varphi_t$ then $\rho^2 + r^2$ is an eigenvalue of $\Delta_\mathbf{M}$ and for any $f\in\mathrm{C}^\infty(S\mathbf{M})$
	\begin{equation*}
	\mathrm{Res}_{\lambda = -\rho + \mathrm{i}r} \left[ Z_f(\lambda) \right] = \sum_{l = 1}^m \langle \widehat{\mathrm{PS}}_{\varphi_l}, f\rangle ,
	\end{equation*}
	where the sum is over an orthonormal $\mathrm{L}^2$-basis of the $\Delta_\mathbf{M}$-eigenspace with eigenvalue $\rho^2 + r^2$.
\end{theorem}
\begin{proof}
	By \Cref{thm:analyt_zeta} we have $\mathrm{Res}_{\lambda = -\rho + \mathrm{i}r} \left[ Z_f(\lambda) \right] = \mathcal{T}_{-\rho + \mathrm{i} r}(f)$ and by \cite[Corollary~6.1]{Weich.2021} we have $\mathcal{T}_{-\rho + \mathrm{i}r}(f) = \sum_{l = 1}^m \langle \widehat{\mathrm{PS}}_{\varphi_l}, f\rangle$. Note that the constants $c(\mathrm{i}r)$ appearing in \cite[Corollary~6.1]{Weich.2021} and defined in \cite[Eq.~(6.1)]{Weich.2021} are the normalization factors for $\mathrm{PS}_{\varphi_l}$; the additional factor of $m^{-1}$ appearing in \cite[Corollary~6.1]{Weich.2021} are due to their slightly different definition of $\mathcal{T}_{\lambda_0}$ in \cite[Eq.~(2.1)]{Weich.2021}. Combining these two results proves our claim.
\end{proof}

We would like to compare our result with previously known results obtained with different techniques. In their paper \cite{Zelditch.2007}, Anantharaman and Zelditch proved a similar close connection between slightly different weighted zeta functions and Patterson-Sullivan distributions: For $f\in \mathrm{C}^\infty(S\mathbf{M})$ they define a weighted zeta function via \cite[Eq.~(1.9)]{Zelditch.2007}
\begin{equation} \label{eq_def_zeta_zelditch}
\mathcal{Z}_f(\lambda) \defgr \sum_\gamma \left(\frac{\exp\left( -\lambda T_\gamma \right)}{1 - \exp\left( - T_\gamma \right)} \int_{\gamma^\#} f \right) ,
\end{equation}
where one sums over all closed geodesics $\gamma$. In \cite[Thm.~1.3]{Zelditch.2007} they state that, provided $f$ is \emph{real analytic}, $\mathcal{Z}_f$ continues meromorphically onto $\mathbb{C}$, its poles in $\{0 < \mathrm{Re}(\lambda) < 1\}$ are of the form $\lambda = 1/2 + \mathrm{i}r$ where $1/4 + r^2$ is an eigenvalue of $\Delta_\mathbf{M}$ and the following residue formula holds:
\begin{equation*}
\mathrm{Res}_{\lambda = 1/2 + \mathrm{i} r} \left[ \mathcal{Z}_f(\lambda) \right] = \sum_{\varphi_i:\, \lambda_i = 1/4 + r^2} \langle \widehat{\mathrm{PS}}_{\varphi_i}, f\rangle .
\end{equation*}
They give two different proofs with the first relying on the thermodynamic formalism and the second on representation theory and a version of Selberg's trace formula. See also the later work \cite{Zelditch.2012} for a generalization of the intertwining between Wigner and Patterson-Sullivan distributions to the non-diagonal case.

Note that for hyperbolic surfaces $\mathrm{det}(\mathrm{id} - \mathcal{P}_\gamma) = (1 - \exp(-T_\gamma))( 1 - \exp(T_\gamma))$, i.e. a simple calculation yields the following relation between $\mathcal{Z}_f$ and our weighted zeta:
\begin{equation*}
Z_f(\lambda) = \sum_{n = 1}^\infty \mathcal{Z}_f(\lambda + n) .
\end{equation*}
We can therefore conclude that given an eigenvalue $1/4 + r^2$ of $\Delta_\mathbf{M}$ the value $-1/2 + \mathrm{i}r$ is a Ruelle resonance of the geodesic flow on $S\mathbf{M}$ and
\begin{equation*}
\mathrm{Res}_{\lambda = -1/2 + \mathrm{i}r}\left[ Z_f(\lambda) \right] = \sum_{\varphi_i:\, \lambda_i = \lambda} \langle \widehat{\mathrm{PS}}_{\varphi_i}, f\rangle .
\end{equation*}

In his thesis \cite{Emonds.2014}, Emonds extended the residue formula of \cite{Zelditch.2007} to the case of hyperbolic manifolds of arbitrary dimension. But this result again imposes a significant restriction on the space of test functions, namely that $f$ be $K$-finite, and the proof heavily relies on techniques from representation theory. It appears that the methods of microlocal analysis are better suited for the meromorphic continuation of (weighted) zeta functions.

Let us finally note that our main result also holds for open systems, so we immediately obtain residue formulae for weighted zeta functions of geodesic flows on convex cocompact hyperbolic manifolds in terms of invariant Ruelle distributions. While there exists so far no theory of Patterson-Sullivan distributions for these systems, a quantum classical correspondence has been established on the level of resonances and resonant states \cite{Weich.2018}. We thus conjecture that also in this setting the invariant Ruelle distributions are, in the high frequency limit, asymptotically equivalent to phase space distributions of quantum resonant states. Given the fact that the residues of $Z_f$ can be numerically calculated quite efficiently, this would provide a method to study phase space distributions of quantum resonant states on Schottky surfaces numerically.


\appendix

\section{Numerical Calculation of Invariant Ruelle Distributions\\ (with Sonja Barkhofen)} \label{numerics}

In this appendix we provide a short outlook on a major application of the weighted zeta function developed in the main text, namely the numerical calculation of invariant Ruelle distributions $\mathcal{T}_{\lambda_0}$. In principle one wants to exploit \Cref{thm:no_strictly_convex} and more concretely the relation
\begin{equation*}
\mathcal{T}_{\lambda_0} = \mathrm{Res}_{\lambda = \lambda_0} Z_f(\lambda) .
\end{equation*}
To use this numerically we require an efficient method for calculating (the residues of) the weighted zeta function $Z_f$. While this endeavor is hopeless in the abstract setting of \Cref{thm:resolvent_without_strict_convexity} there are concrete dynamical systems where this calculation is possible due to the availability of a symbolic encoding of the dynamics. Two such system which are well-known in the literature are \emph{convex obstacle scattering} and \emph{geodesic flows on convex cocompact hyperbolic surfaces}. For the sake of brevity we will focus on a particular instance of the former class of systems, namely so-called \emph{symmetric $3$-disc systems}.

This appendix is organized as follows: We begin by giving a short introduction to $3$-disc systems in Section \ref{numerics.1}. With this setup at hand we then provide some first numerical results in Section \ref{numerics.2}.

\subsection{Introducing $3$-Disc Systems} \label{numerics.1}

The $3$-disc system is a paradigmatic example of a convex obstacle scattering dynamics \cite{Ikawa.1988, GR89cl}. It is given by three discs $\mathrm{D}(x_i, r_i)\subseteq \mathbb{R}^2$, $i\in \{1, 2, 3\}$, with radii $r_i > 0$, centers $x_i\in \mathbb{R}^2$ and disjoint closures. The dynamics takes place on the unit sphere bundle $S\left(\mathbb{R}^2\setminus \bigcup_{i = 1}^3 \mathring{\mathrm{D}}(x_i, r_i)\right)$: In the interior $S\left(\mathbb{R}^2\setminus \bigcup_{i = 1}^3 \mathring{\mathrm{D}}(x_i, r_i)\right)$ its trajectories coincide with the Euclidean geodesic flow, and upon boundary intersection the trajectories experience specular reflections. We will only consider fully \emph{symmetric} $3$-disc systems here, and for our purposes these are uniquely described by the quotient $d / r$ of the common radius $r = r_1 = r_2 = r_3$ and the side length $d > 0$ of the equilateral triangle on which the centers of the discs are positioned. For a graphical illustration of this setup see Figure \ref{fig1}.

\begin{figure}[H]
\includegraphics[scale=0.7]{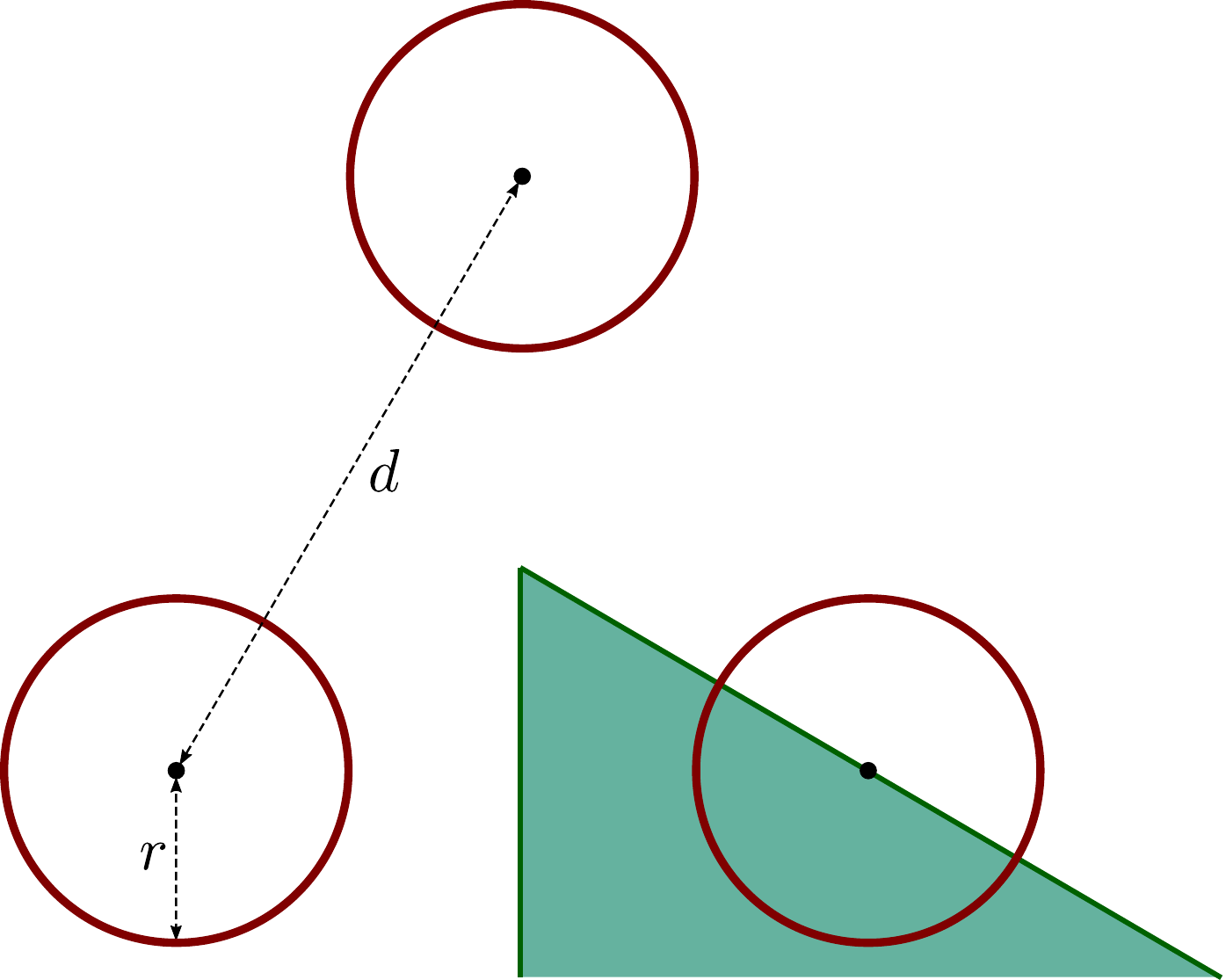}
\caption{A symmetric $3$-disc system with its defining paramters $r$ and $d$. The fundamental domain is given by the green region.}
\label{fig1}
\end{figure}

Now the dynamics just described is obviously not smooth because of the instantaneous boundary reflections. If we assume that $d / r$ is sufficiently large then this lack of smoothness is properly separated from the trapped set, though. In this case it can be dealt with via smooth models. To keep our presentation short we refer the reader to \cite{schuette.2021} where the construction of smooth models as well as the meromorphic continuation of weighted zeta functions was carried out in detail and in a more general setting.

What we will actually investigate numerically is not quite the $3$-disc dynamics just described, but rather a symmetry reduced variant. A fundamental domain of this symmetry reduction is shown in green in Figure \ref{fig1}. The reduced dynamics still admits a symbolic coding and is well understood on the classical as well as on the quantum side. It is used here as it provides the model of choice for experimental realizations \cite{PhysRevE.86.066205,PhysRevLett.110.164102}. For more details and references see \cite{Weich2014}.

Our numerical algorithm itself resembles the algorithm developed by Cvitanovic and Eckhardt \cite{Eckhardt.1989,cvitanovic1991periodic} in physics and Jenkinson and Pollicott \cite{jenkinson2002calculating} or Borthwick \cite{Borthwick.2014} in mathematics: One can derive a cycle expansion for the weighted zeta function $Z_f$ associated with a given $3$-disc system. To calculate concrete summands in this expansion we make use of the symbolic encoding of closed trajectories available for sufficiently large $d / r$. A detailed description of the algorithm will be presented elsewhere \cite{Schuette.2022}. We just want to mention the following two central simplifications:

\begin{enumerate}
	\item To be able to plot the distributions $\mathcal{T}_{\lambda_0}$ we calculate their convolution with Gaussians with variance $\sigma > 0$. In the limit $\sigma \rightarrow 0$ this convolution converges to $\mathcal{T}_{\lambda_0}$ in $\mathcal{D}'$ and it is reasonable to expect the numerical results for small but positive $\sigma$ to reveal interesting properties of $\mathcal{T}_{\lambda_0}$ itself.
	\item While the convolutions discussed in (1) are smooth they still live on the $3$-dimensional state space of the $3$-disc system. To obtain $2$-dimensional plots we restrict $\mathcal{T}_{\lambda_0}$ to a Poincar\'{e} section $\Sigma$ via \Cref{lem:thm_ruelle_restriction}.
\end{enumerate}

For the numerics presented below we used a specific Poincar\'{e} section $\Sigma\subseteq S\mathbb{R}^2$ defined by so-called \emph{Birkhoff coordinates} as follows: First, fix one of the discs and an origin on the boundary of this disc. Then a point $(q, p)\in [-\pi, \pi]\times [-1, 1]$ corresponds to the point $(x, v)\in S\mathbb{R}^2$ such that the boundary arc connecting the origin and $x$ has length $q\cdot r$ and such that the projection $\langle v, t(x)\rangle$ of $v$ onto the tangent $t(x)$ to the disc at $x$ equals $p$. For an illustration of these coordinates see Figure \ref{fig2}. Now the transversality condition of \Cref{lem:thm_ruelle_restriction} is obviously satisfied, making the restriction $\mathcal{T}_{\lambda_0}\big|_{\Sigma}$ a well-defined distribution. It is this object which will be plotted numerically in the following section.

\begin{figure}[H]
	\includegraphics[scale=0.7]{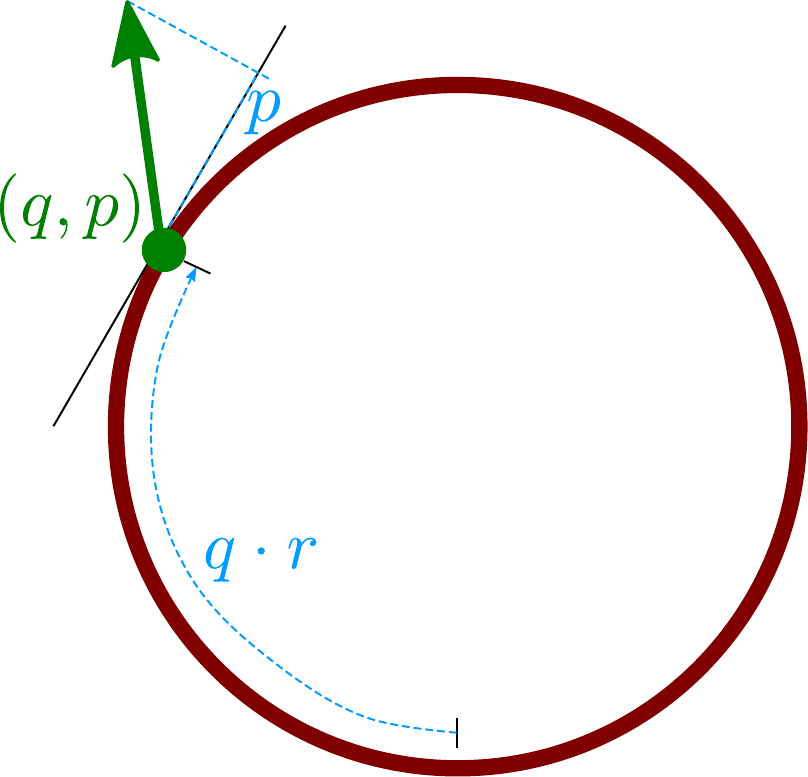}
	\caption{The Birkhoff coordinates for the Poincar\'{e} section $\Sigma \subseteq S\mathbb{R}^2$ of a symmetric $3$-disc system used in the numerics below.}
	\label{fig2}
\end{figure}

\subsection{Proof-of-Principle Results} \label{numerics.2}

In this section we present first numerical calculations of invariant Ruelle distributions. As already mentioned we restrict the distributions to the Poincar\'{e} section $\Sigma$ and then approximate by convolution with Gaussians of width $\sigma$. As the trapped set $K$ itself is fractal and therefore hard to visualize, we chose the following alternative: Denote by $K_1 \supseteq K$ those points of phase space which experience at least one disc reflection either in forward or backward time. We included the intersection $\Sigma_1\defgr K_1\cap \Sigma$ in the figures below to give an idea of where the invariant Ruelle distributions are supposed to be supported in theory.

We begin the first series of illustrations by plotting four example resonances and associated distributions along the first clearly distinguishable resonance chain. In particular, we begin with the point closest to the spectral gap and continue towards the intersection with the second distinct chain. Our choices are marked in red in the first row of Figure \ref{fig3}.

In the second and third rows of Figure \ref{fig3} we plotted the distribution associated with the marked resonance and the two choices $\sigma = 0.1$ and $\sigma = 0.001$, respectively. Going from left to right in either the first or the second row shows that the invariant Ruelle distributions clear encode some kind of information regarding the location of their associated resonances. Especially the distribution at the point of intersection of the first two chains (fourth column in Figure \ref{fig3}) differs significantly from the first three, which only exhibit a gradual reduction of intensity in the left and right component of $\Sigma_1$.

Going from top to bottom in Figure \ref{fig3} we see how the reduction of $\sigma$ by two orders of magnitude significantly increases the localization of the distributions on the trapped set. This behavior is expected by the theory developed above and could allow a detailed numerical investigation of invariant Ruelle distributions on successively finer scales of the fractal trapped set.

\begin{figure}[h]
	\includegraphics[page=1, width=1.\textwidth, trim={0cm 0cm 0cm 0cm}]{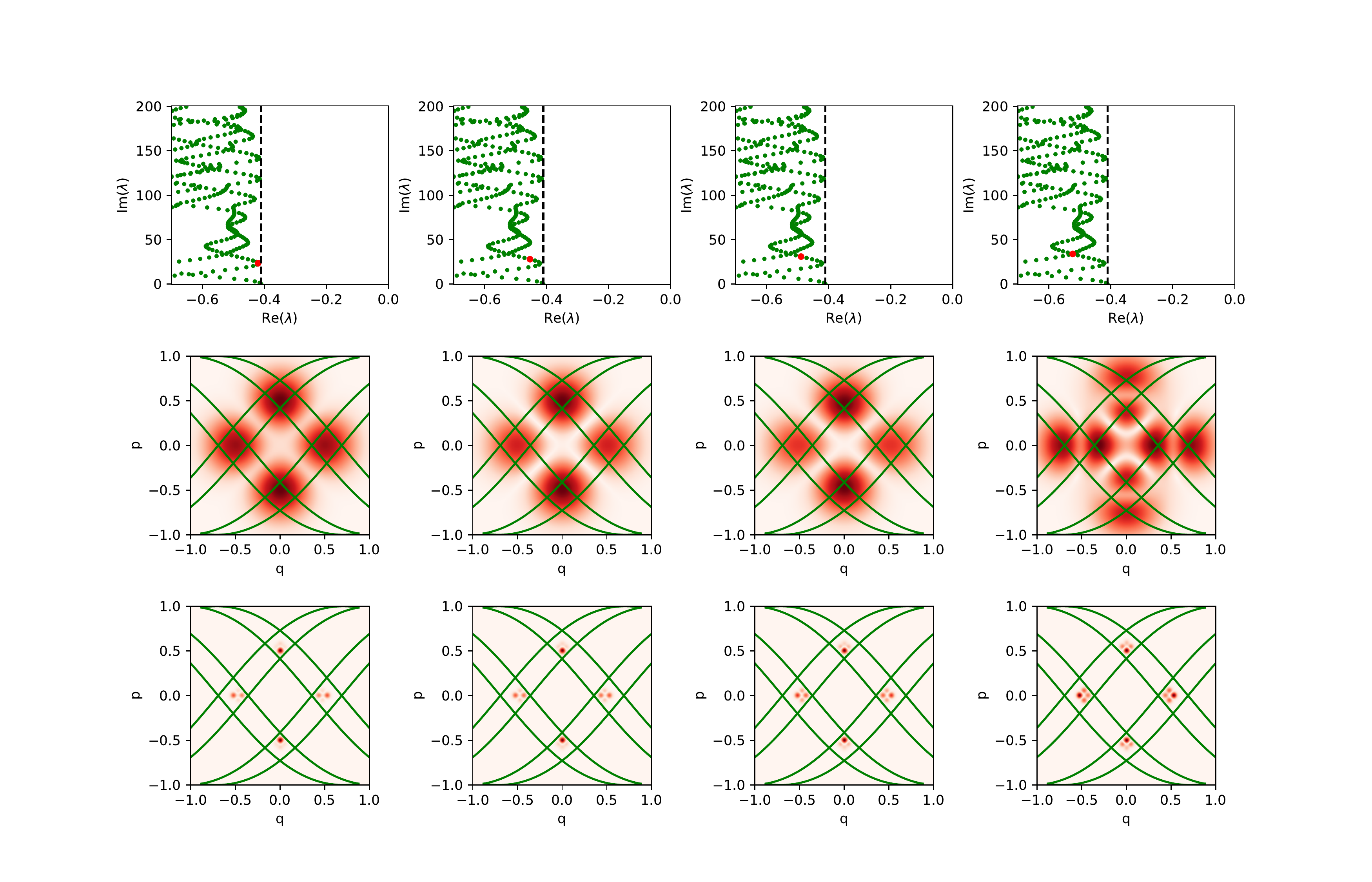}
	\caption{Comparison of invariant Ruelle distributions along a resonance chain for $d / r = 6$. The first row highlights in red the resonance at which the distributions in the second and third rows was evaluated. The second row was computed with $\sigma = 0.1$ and the third row with $\sigma = 0.001$. One clearly recognizes dynamical changes in $\mathcal{T}_{\lambda_0}$ as $\lambda_0$ varies along the chain. In addition, the reduction of $\sigma$ seems to further localize the distribution on the trapped set, which theoretically contains the support of any $\mathcal{T}_{\lambda_0}$.}
	\label{fig3}
\end{figure}

Our second series of invariant Ruelle distributions is meant to give a first impression of a curious phenomenon which has not been understood theoretically yet: Calculating numerically resonances with imaginary part up to about $750$ it would appear that the maximal real part which occurs becomes successively smaller the larger $\mathrm{Im}(\lambda)$ becomes. If we proceed to even larger imaginary parts this progression reverses and at about $\mathrm{Im}(\lambda) = 1500$ we observe several resonances with real parts close to the theoretical maximum of $\lambda_1$, where $\lambda_1$ denotes the first resonance on the real line. These observations are shown in the first row of Figure \ref{fig4}. The same effect has been observed in even more pronounced fashion for resonances on Schottky surfaces \cite{borthwick2016symmetry}.

As the question of asymptotic spectral gaps for such open systems is an important unsolved problem, it is interesting to understand such a recurrence of resonances to a neighborhood of the critical line. We therefore calculated the invariant Ruelle distributions for those resonances close to the critical line in the second and third rows of Figure \ref{fig4}: The distribution plots from top left to bottom right belong to the resonances marked in red and ordered from small to large imaginary part. We immediately notice that all eight distributions while associated with different resonances appear exceedingly similar, even though we already calculated them with the rather small value of $\sigma = 0.001$. We have to admit that we cannot explain this observation so far but find it quite remarkable.

\begin{figure}[h]
	\includegraphics[page=1, width=1.\textwidth]{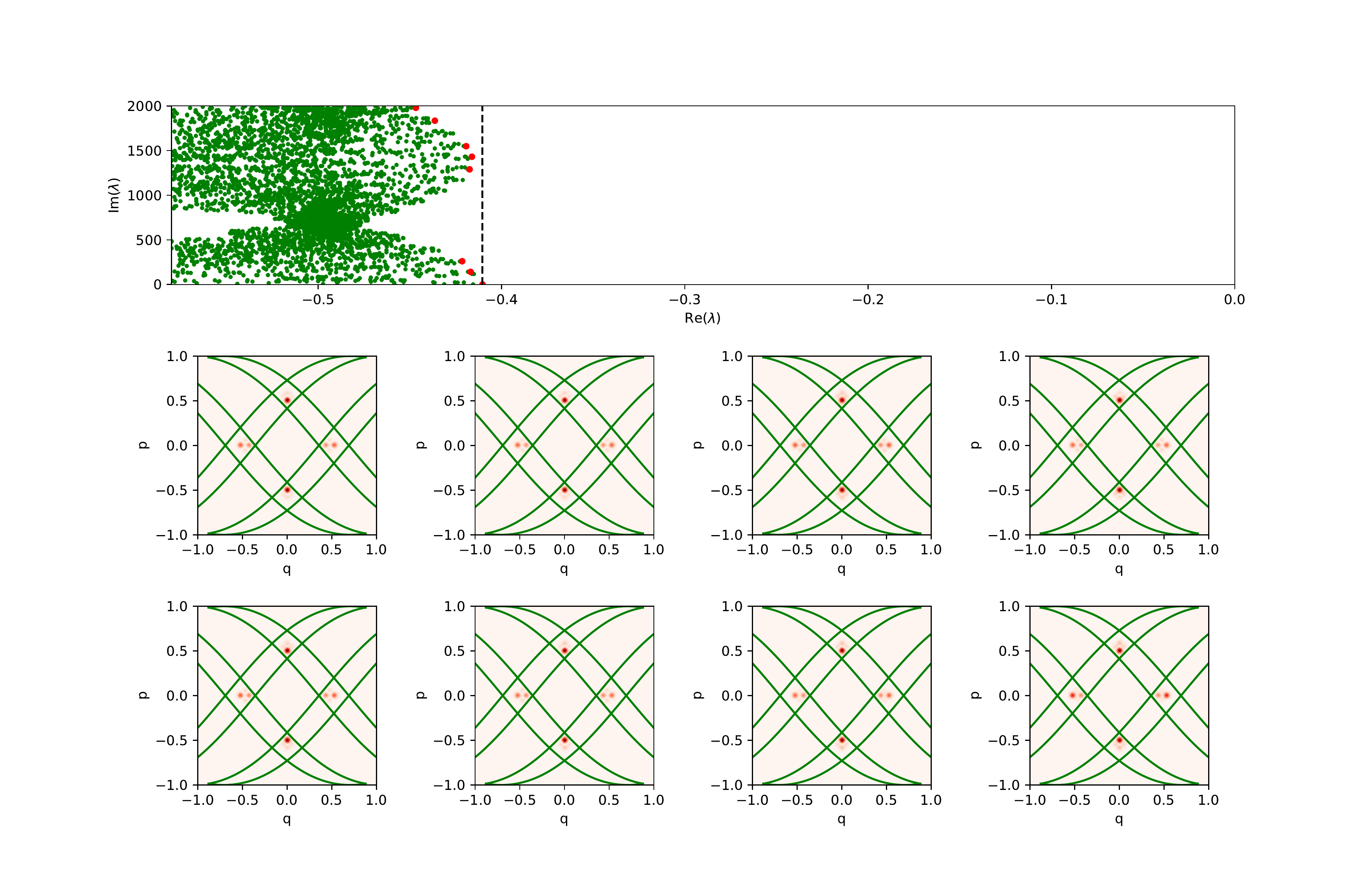}
	\caption{Comparison of invariant Ruelle distributions for several different resonances near the line $\mathrm{Re}(\lambda) = \lambda_1$, where $\lambda_1$ denotes the resonances with maximal real part. The resonances marked in red (from bottom to top) correspond to the plotted distributions (from top left to bottom right). Throughout we have $\sigma = 0.001$. Note how the distributions all appear very similar, even on this second level of the fractal trapped set.}
	\label{fig4}
\end{figure}

\begin{minipage}{0.7\textwidth}
Additional and more detailed illustrations can be found on the supplementary website \url{https://go.upb.de/ruelle}. In particular, it contains several additional distributions along the first chain discussed above, illustrations of further resonances near the spectral gap, and plots along a second resonance chain.\\
\end{minipage}
\begin{minipage}{0.25\textwidth}
	\centering
	\includegraphics{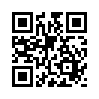}\\
	\centering
	\tiny\url{https://go.upb.de/ruelle}
\end{minipage}


\bibliographystyle{amsalpha}
\bibliography{bibo}
\bigskip

\end{document}